\documentclass[a4paper,12pt]{article}

\usepackage{geometry}
\usepackage[T1]{fontenc}
\usepackage{lmodern}

\usepackage{amsmath}
\usepackage{amssymb}
\usepackage{amsthm}

\usepackage{mathrsfs}

\usepackage{braket}
\usepackage{graphicx}

\usepackage{enumerate}

\usepackage[abbrev]{amsrefs}

\newcommand{\R}{\mathbb{R}}

\newcommand{\N}{\mathbb{N}}
\newcommand{\K}{\mathcal{K}}

\newcommand{\conv}{\operatorname{conv}}
\newcommand{\pos}{\operatorname{pos}}

\newcommand{\interior}{\operatorname{int}}

\usepackage{bm}

\newcommand{\va}{\bm{a}}
\newcommand{\vb}{\bm{b}}
\newcommand{\vc}{\bm{c}}
\newcommand{\vd}{\bm{d}}

\newcommand{\vn}{\bm{n}}
\newcommand{\vo}{\bm{o}}
\newcommand{\vp}{\bm{p}}
\newcommand{\vq}{\bm{q}}
\newcommand{\vr}{\bm{r}}

\newcommand{\vu}{\bm{u}}

\newcommand{\vx}{\bm{x}}
\newcommand{\vy}{\bm{y}}
\newcommand{\vz}{\bm{z}}
\newcommand{\vxi}{\bm{\xi}}

\newcommand{\checkK}{\check{\mathcal{K}}}

\theoremstyle{plain}
\newtheorem{theorem}{Theorem}[section]
\newtheorem{lemma}[theorem]{Lemma}
\newtheorem{proposition}[theorem]{Proposition}
\newtheorem{corollary}[theorem]{Corollary}

\theoremstyle{definition}
\newtheorem{remark}[theorem]{Remark}

\newtheorem*{conjecture}{Conjecture}
\newtheorem*{problem}{Problem}

\usepackage{xcolor}

\newcommand{\gray}[1]{{\color{gray}#1}}

\usepackage{ulem}
\usepackage{cancel}

\usepackage{multirow}

\renewcommand{\Diamond}{\includegraphics[width=0.8em]{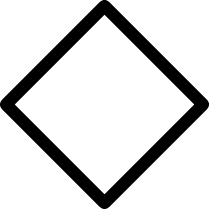}}
\newcommand{\Pentagon}{\includegraphics[width=0.8em]{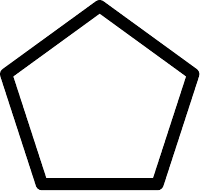}}

\begin{document}

\title{Minimal volume product of three dimensional convex bodies with various discrete symmetries}

\author{Hiroshi Iriyeh
\thanks{Graduate School of Science and Engineering, Ibaraki University, \newline \qquad 
e-mail: hiroshi.irie.math@vc.ibaraki.ac.jp}
\and 
Masataka Shibata
\thanks{Department of Mathematics, Tokyo Institute of Technology, \newline \qquad 
e-mail: shibata@math.titech.ac.jp}
}

\maketitle

\begin{abstract}
We give the sharp lower bound of the volume product of three dimensional convex bodies which are
invariant under a discrete subgroup of $O(3)$ in several cases.
We also characterize the convex bodies with the minimal volume product in each case.
In particular, this provides a new partial result of the non-symmetric version of Mahler's conjecture
in the three dimensional case.
\end{abstract}

\section{Introduction and main results}

\subsection{Mahler's conjecture and its generalization}

Let $K$ be a {\it convex body} in $\R^n$, i.e, $K$ is a compact convex set in $\R^n$ with nonempty interior. 
Then $K \subset \R^n$ is said to be {\it centrally symmetric} if it satisfies that $K=-K$.
Denote by $\K^n$ the set of all convex bodies in $\R^n$ equipped with the Hausdorff metric and by $\K^n_0$ the set of all $K \in \K^n$ which are centrally symmetric.

The interior of $K \in \K^n$ is denoted by $\interior K$.
For a point $z \in \interior K$, the {\it polar body} of $K$ with respect to $z$ is defined by
\begin{equation*}
K^z =
\left\{
y \in \R^n;
(y-z) \cdot (x-z) \leq 1 \text{ for any } x \in K
\right\},
\end{equation*}
where $\cdot$ denotes the standard inner product on $\R^n$.
Then an affine invariant
\begin{equation}
\label{eq:a}
 \mathcal{P}(K) := \min_{z \in \interior K}|K| \, |K^z|
\end{equation}
is called {\it volume product} of $K$, where $|K|$ denotes the $n$-dimensional volume of $K$ in $\R^n$.
It is known that for each $K \in \K^n$ the minimum of \eqref{eq:a} is attained at the unique point $z$ on $K$, which is called {\it Santal\'o point} of $K$ (see, e.g., \cite{MP}).
For a centrally symmetric convex body $K \in \K^n_0$, the Santal\'o point of $K$ is the origin $o$.
See \cite{Sc}*{p.\,546} for further information.
In the following, the polar of $K$ with respect to $o$ is denoted by $K^\circ$.

Mahler's conjecture \cite{Ma57} states that for any $K \in \K^n_0$,
\begin{equation}
\label{eq:b}
 \mathcal{P}(K) \geq  \frac{4^n}{n!}
\end{equation}
would be hold.
He proved it in the case where $n=2$ (\cite{Ma59}).
The three dimensional case recently proved in \cite{IS} (see also \cite{FHMRZ} for a nice simple proof of an equipartition result used in \cite{IS}).
Although the case that $n\geq4$ is still open, $n$-cube or, more generally, Hanner polytopes satisfy the equality in \eqref{eq:b} and are predicted as the minimizers of $\mathcal{P}$.
As for non-symmetric bodies there is another well-known conjecture for the lower bound of the volume product as follows.

\begin{conjecture}
Any $K \in \K^n$
satisfies that
\begin{equation*}
\mathcal{P}(K) \geq \frac{(n+1)^{n+1}}{(n!)^2}
\end{equation*}
with equality if and only if $K$ is a simplex.
\end{conjecture}

This was proved by Mahler for $n=2$ \cite{Ma59} (see also \cite{Me}).
Note that this conjecture remains open for $n \geq 3$, see e.g., \cite{BF}, \cite{FMZ}, \cite{KR}.

On the other hand, Barthe and Fradelizi (\cite{BF}) obtained the sharp lower bound of $\mathcal{P}(K)$ when $K$ has many reflection symmetries, in other words, when $K$ is invariant under the action of a Coxeter group.
It is worthwhile to estimate the volume product of the $G$-invariant convex body $K \subset \mathbb R^n$ from below for more general discrete subgroups $G$ of the orthogonal group $O(n)$.

\begin{problem}
Let $G$ be a discrete subgroup of $O(n)$.
Denote by $\mathcal{K}^n(G)$ the set of all convex bodies $K \in \mathcal{K}^n$ which satisfy that $K = g(K)$ for any $g \in G$.
Then, consider the minimizing problem
\begin{equation*}
\min_{K \in \mathcal{K}^{n}(G)} \mathcal{P}(K)
\end{equation*}
and determine all the minimizers $K \in \mathcal{K}^n(G)$.
\end{problem}

From this viewpoint, Mahler's conjecture corresponds to the case $G=\{E, -E\} \cong {\mathbb Z}_2$ ($E$ is the identity matrix) and the non-symmetric case corresponds to $G=\{E\}$.
Following \cite{BF} for a subset $A \subset \R^n$ we recall the subgroup of $O(n)$ defined by
\begin{equation*}
O(A)=\{
g \in O(n); g(A)=A
\}.
\end{equation*}
Similarly, we can define $SO(A) \subset SO(n)$.
For example, if $P \subset \R^n$ is a regular polytope with $o$ as the centroid, then $O(P)$ is a discrete subgroup of $O(n)$.
Indeed, Problem was solved for $G=O(P)$ for every regular polytope $P \subset \R^n$ (\cite{BF}*{Theorem 1 (i)}).
Apart from the setting of \cite{BF}, we can settle the case $G=SO(P) \subset SO(n)$ as well.

In this paper, we shall focus on the three dimensional case of the above problem and consider $G$-invariant convex bodies $K \subset \mathbb R^3$ for a discrete subgroup $G$ of $O(3)$.
Note that the classification of the discrete subgroups of $O(3)$ is well-known.
Before we explain the details, we recall the two dimensional case in which the problem has already been solved.

\subsection{The two dimensional results}

For $\ell \in \N$, we put $\xi=2 \pi/\ell$ and
\begin{equation*}
R_\ell:=
\begin{pmatrix}
\cos \xi & -\sin \xi \\
\sin \xi &  \cos \xi 
\end{pmatrix}, \quad
V:=
\begin{pmatrix}
1 & 0 \\
0 & -1 \\
\end{pmatrix}.
\end{equation*}
Up to conjugation, all the discrete subgroups of $O(2)$ are
\begin{equation*}
C_\ell=\braket{R_\ell} \text{ and } D_\ell=\braket{R_\ell, V} \quad (\ell \in \N).
\end{equation*}
Note that $\mathcal{K}^2(C_1)=\mathcal{K}^2$ and $\mathcal{K}^2(C_2)=\mathcal{K}^2_0$.
We denote by $\Pentagon_\ell$ the regular $\ell$-gon with $o$ as the centroid.
For simplicity, we also denote $\triangle:=\Pentagon_3$ and $\square:=\Pentagon_4$.

The following is a summary of the known results.

\begin{theorem}[\cites{Ma59, BF, BMMR}]
\begin{enumerate}[\upshape (i)]
 \item 
$\mathcal{P}(K) \geq \mathcal{P}(\triangle)$ holds for $K \in \mathcal{K}^2(C_1)$ with equality if and only if $K$ is the image of $\triangle$ by an affine transformation of $\R^2$.
The same inequality holds for $G=D_1$.
 \item 
$\mathcal{P}(K) \geq \mathcal{P}(\square)$ holds for $K \in \mathcal{K}^2(C_2)$ with equality if and only if $K$ is the image of $\square$ by a linear transformation of $\R^2$.
The same inequality holds for $G=D_2$.
 \item 
Assume that $\ell \geq 3$. 
Then $\mathcal{P}(K) \geq \mathcal{P}(\Pentagon_\ell)$ holds for $K \in \mathcal{K}^2(C_\ell)$ with equality if and only if $K$ is a dilation or rotation of $\Pentagon_\ell$.
The same inequality holds for $G=D_\ell$.
\end{enumerate}
\end{theorem}

\subsection{The three dimensional results}
\label{sec:1.3}

We shall state known results and the main theorem in the case where $n=3$.
Let us first recall the list of all discrete subgroups of $O(3)$ by using Schoenflies' notation.
We equip $\R^3$ with the usual orthogonal $xyz$-axies and put
\begin{equation*}
R_\ell:=
\begin{pmatrix}
 \cos \xi & - \sin \xi & 0 \\
 \sin \xi & \cos \xi & 0 \\
 0 & 0 & 1
\end{pmatrix}, 
V:=
\begin{pmatrix}
 1 &  0 & 0 \\
 0 & -1 & 0 \\
 0 &  0 & 1
\end{pmatrix}, 
H:=
\begin{pmatrix}
 1 & 0 &  0 \\
 0 & 1 &  0 \\
 0 & 0 & -1
\end{pmatrix},
\end{equation*}
where $\ell \in \N$, $\xi:=2\pi/\ell$.
Here $R_\ell$, $V$, and $H$ mean the rotation through the angle $2\pi/\ell$ about the $z$-axis, the reflection with respect to the $zx$-plane, and the reflection with respect to the $xy$-plane, respectively.
It is known that, up to conjugation, the discrete subgroups of $O(3)$ are classified as the seven infinite families
\begin{equation*}
\begin{aligned}
& C_\ell:=\braket{R_\ell}, \quad
C_{\ell h}:=\braket{R_\ell,H}, \quad
C_{\ell v}:=\braket{R_\ell,V}, \quad
S_{2\ell}:=\braket{R_{2\ell}H}, \\
& D_\ell:=\braket{R_\ell,VH}, \quad
D_{\ell d}:=\braket{R_{2\ell}H,V}, \quad
D_{\ell h}:=\braket{R_\ell,V,H},
\end{aligned}
\end{equation*}
and the following seven finite groups (see e.g., \cite{CS}*{Table 3.1}):
\begin{equation*}
\begin{aligned}
T&:=\left\{g \in SO(3); g \triangle =\triangle \right\}, &
T_d&:=\left\{g \in O(3); g \triangle =\triangle \right\}, \quad
T_h:= \left\{\pm g ; g \in T\right\}, \\
O&:=\left\{g \in SO(3); g \Diamond =\Diamond \right\}, &
O_h&:=\left\{g \in O(3); g \Diamond =\Diamond \right\}= \left\{\pm g ; g \in O\right\}, \\
I&:=\left\{g \in SO(3); g \Pentagon =\Pentagon \right\}, &
I_h&:=\left\{g \in O(3); g \Pentagon =\Pentagon \right\}= \left\{\pm g ; g \in I\right\},
\end{aligned}
\end{equation*}
where $\triangle=\triangle^3$, $\Diamond=\Diamond^3$, and $\Pentagon=\Pentagon^3$ denote the regular tetrahedron (simplex), the regular octahedron, and the regular icosahedron with $o$ as their centroids, respectively.
Note that, for example, if we fix the configuration of $\triangle$, then $T_d$ is uniquely determined as the subgroup of $O(3)$ without any ambiguity of conjugations.
In other words, any discrete subgroup conjugated to $T_d$ in $O(3)$ is realized as $\left\{g \in O(3); g \triangle' =\triangle' \right\}$, where $\triangle' =k \triangle$ for some $k \in O(3)$.
Using the notation of \cite{BF}, $T_d=O(\triangle)\cong O(\triangle')$.
We also note that among the above classification $C_\ell, D_\ell, T, O, I$ are indeed subgroups of $SO(3)$, and that the Santal\'o point of $K \in \mathcal{K}^3(G)$ is the origin $o$ except only two subgroups $G=C_\ell$ or $C_{\ell v}$ $(\ell \in \N)$.

Now we recall the known results of the three dimensional case.
A convex body $K \in \mathcal{K}^3(D_{2h})$ is $1$-unconditional and $D_{2h} \cong ({\mathbb Z}_2)^3$ as groups.

\begin{theorem}[\cite{SR}, \cite{Me1986}]
\label{thm:SR}
For $K \in \mathcal{K}^3(D_{2h})$, we have
\begin{equation*}
 \mathcal{P}(K) \geq \mathcal{P}(\Diamond) =\frac{32}{3}.
\end{equation*}
The equality holds if and only if $K$ is a three dimensional Hanner polytope \textup{(\cite{Re}, \cite{Me1986})}, i.e., $K$ is the image of the regular octahedron $\Diamond$ or the cube $\Diamond^\circ$ by a linear transformation of $\R^3$ by a diagonal matrix.
\end{theorem}

\begin{theorem}[\cite{BF}*{Theorem 1}]
\label{thm:BF}
\begin{enumerate}[{\upshape (i)}]
\item 
$\mathcal{P}(K) \geq \mathcal{P}(\triangle)$ holds for $K \in \mathcal{K}^3(T_d)$.
The equality holds if and only if $K$ is a dilation of the regular tetrahedron $\triangle$ or  $\triangle^\circ$.
 \item 
$\mathcal{P}(K) \geq \mathcal{P}(\Diamond)$ holds for $K \in \mathcal{K}^3(O_h)$.
The equality holds if and only if $K$ is a dilation of the regular octahedron $\Diamond$ or the cube $\Diamond^\circ$.
 \item 
$\mathcal{P}(K) \geq \mathcal{P}(\Pentagon)$ holds for $K \in \mathcal{K}^3(I_h)$.
The equality holds if and only if $K$ is a dilation of the regular icosahedron $\Pentagon$ or the regular dodecahedron $\Pentagon^\circ$.
 \item 
Assume that $\ell \geq 3$. 
Then $\mathcal{P}(K) \geq \mathcal{P}(P_\ell)$ holds for $K \in \mathcal{K}^3(D_{\ell h})$,
where $P_\ell$ is the $\ell$-regular right prism defined by
\begin{equation*}
P_\ell :=
\conv
\left\{
\begin{pmatrix}
\cos k\xi \\
\sin k\xi \\
1
\end{pmatrix},
\begin{pmatrix}
\cos k\xi \\
\sin k\xi \\
-1
\end{pmatrix}; k=0,\ldots,l-1
\right\}.
\end{equation*}
Here, we denote by $\mathrm{conv}\,S$ the convex hull of a set $S$.
\end{enumerate}
\end{theorem}

Furthermore, the solution of Mahler's conjecture for $n=3$ \cite{IS}*{Theorem 1} is corresponding to 
the case $G=S_2 (\cong {\mathbb Z}_2)$, which immediately implies

\begin{corollary}
\label{cor:1}
Assume that $G=C_{2h}$, $T_h$, $S_6$, $D_{3d}$, or $S_2$.
For $K \in \mathcal{K}^3(G)$, we have
\begin{equation*}
 \mathcal{P}(K) \geq \mathcal{P}(\Diamond) =\frac{32}{3}
\end{equation*}
with equality if and only if $K$ is a linear image of $\Diamond$ or $\Diamond^\circ$ which is invariant under the group $G$.
\end{corollary}

The following is the main result of this paper.

\begin{theorem}
\label{thm:2}
\begin{enumerate}[{\upshape (i)}]
\item 
$\mathcal{P}(K) \geq \mathcal{P}(\triangle)$ holds for $K \in \mathcal{K}^3(T)$.
The equality holds if and only if $K$ is a dilation of the regular tetrahedron $\triangle$ or $\triangle^\circ$.
 \item 
$\mathcal{P}(K) \geq \mathcal{P}(\Diamond)$ holds for $K \in \mathcal{K}^3(O)$.
The equality holds if and only if $K$ is a dilation of the regular octahedron $\Diamond$ or the cube $\Diamond^\circ$.
 \item 
$\mathcal{P}(K) \geq \mathcal{P}(\Pentagon)$ holds for $K \in \mathcal{K}^3(I)$.
The equality holds if and only if $K$ is a dilation of the regular icosahedron $\Pentagon$ or the regular dodecahedron $\Pentagon^\circ$.
 \item
Assume that $\ell \geq 3$. 
Then $\mathcal{P}(K) \geq \mathcal{P}(P_\ell)$ holds for $K \in \mathcal{K}^3(C_{\ell h})$,
where $P_\ell$ is the $\ell$-regular right prism defined in Theorem \ref{thm:BF}.
The equality holds if and only if
$K$ coincides with $P_\ell$ or $P_\ell^\circ$ up to a linear transformation in the abelian subgroup of $GL(3,\mathbb R)$ defined by
\begin{equation*}
\mathcal{G}:=
\left\{
\begin{pmatrix}
a\cos\theta & -a\sin\theta & 0 \\
a\sin\theta & a\cos\theta & 0 \\
0 & 0 & b 
\end{pmatrix}; a,b>0
\right\}.
\end{equation*}
The same inequality holds for $K \in \mathcal{K}^3(D_\ell)$.
The equality holds if and only if
$K$ coincides with $P_\ell$ or $P_\ell^\circ$ up to a linear transformation in 
\begin{equation*}
\mathcal{G'}:=
\left\{
\begin{pmatrix}
a & 0 & 0 \\
0 & a & 0 \\
0 & 0 & b 
\end{pmatrix}; a,b>0, \theta \in \R
\right\}.
\end{equation*}
\end{enumerate}
\end{theorem}
In particular, Theorem \ref{thm:2} (i) provides a new partial result of the non-symmetric version of Mahler's conjecture in the three dimensional case (cf.\ Theorem \ref{thm:BF} (i)).
The results from Theorem \ref{thm:SR} to Theorem \ref{thm:2} are summarized as follows.

\begin{center}
\begin{tabular}{|c||c|c|c|c|c|c|c|}\hline
 $\ell$ & $1$ & $2$ & $3$ & $4$ & $5$ & $6$ & $\cdots$ \\ \hline
 $C_{\ell}$ & $\{E\}$ & $C_{2}$ & $C_{3}$ & $C_{4}$ & $C_{5}$ & $C_{6}$
 & $\cdots$ \\ \hline
 $C_{\ell v}$ & $C_{1v}$ & $C_{2v}$ & $C_{3v}$ & $C_{4v}$ & $C_{5v}$ &
 $C_{6v}$ & $\cdots$ \\ \hline
 $C_{\ell h}$ & \gray{$C_{1v}$} & {$\bullet C_{2h}$} & {$\circ C_{3h}$}
 & {$\circ C_{4h}$} & {$\circ C_{5h}$} & {$\circ C_{6h}$} & {$\circ
 \cdots$} \\ \hline
 $S_{\ell}$ & \gray{$C_{1v}$} & {$\bullet S_{2}$} & \gray{$C_{3h}$} &
 $S_{4}$ & \gray{$C_{5h}$} & $\bullet S_{6}$ & $\cdots$ \\ \hline
 $D_{\ell}$ & \gray{$C_{2}$} & $D_{2}$ & {$\circ D_{3}$} & {$\circ
 D_{4}$} & {$\circ D_{5}$} & {$\circ D_{6}$} & {$\circ \cdots$} \\
 \hline
 $D_{\ell h}$ & \gray{$C_{2v}$} & {$\ddagger D_{2h}$} & {$\dagger
 D_{3h}$} & {$\dagger D_{4h}$} & {$\dagger D_{5h}$} & {$\dagger
 D_{6h}$} & {$\dagger \cdots$} \\ \hline
 $D_{\ell d}$ & \gray{$C_{2h}$} & $D_{2d}$ & {$\bullet D_{3d}$} &
 $D_{4d}$ & $D_{5d}$ & $D_{6d}$ & $\cdots$ \\ \hline\hline
  & {$\circ T$} & {$\dagger T_d$} & {$\bullet T_h$} & {$\circ O$} &
 {$\dagger O_h$} & {$\circ I$} & {$\dagger I_h$} \\ \hline
\end{tabular}

\medskip

\begin{tabular}{ll}
$\dagger$: The results in \cite{BF} & $\ddagger$: The result in \cite{SR} \\
\multicolumn{2}{l}{$\bullet$: Results deduced from \cite{IS} (The case where $S_2 \subset G$ and $\Diamond$ is a minimizer)} \\
$\circ$: New results (Theorem \ref{thm:2}) & \gray{Gray}: Duplicates \\
\end{tabular}
\end{center}

\subsection{Method of the proof and organization of this paper}

The proof of results in this paper is based on a simple inequality \cite{IS}*{Proposition 3.2} that gives rise to an effective method to estimate the volume product $\mathcal{P}(K)$ from below, which is a natural extension of the volume estimate by M. Meyer.
He gave in \cite{Me1986} an elegant proof of the result by Saint-Raymond \cite{SR} (see Theorem \ref{thm:SR} above for $n=3$) for $1$-unconditional bodies $K \subset \R^n$ as follows. 
An $1$-unconditional body $K$ is determined by the part $K_1$ in the first octant:
\begin{equation*}
K_1 := K \cap \{ (x_1,\ldots,x_n) \in \R^n; x_1 \geq 0,\cdots,x_n \geq 0 \}.
\end{equation*}
We denote the intersection of $\partial K$ with each positive part of coordinate axes by $P_1,\ldots,P_n$, respectively.
For any point $P$ in $K_1$, we can make the polytope which is the convex full of vertices $o$, $P_1,\ldots,P_n$, and $P$.
The volume of this polytope is less than or equal to that of $K_1$.
This inequality yields a test point contained in the dual (truncated) cone of $K_1$.
Interchanging the role of $K_1$ and the dual cone, we get another test point in $K_1$.
By paring these two test points, we obtain the sharp lower bound estimate of $\mathcal{P}(K)$ for $1$-unconditional bodies $K$.

Meyer's estimate for the above $K_1$ can be generalized to the case of a convex truncated cone and it was used in \cite{BF}.
Actually, $n$-dimensional convex bodies with many hyperplane symmetries were treated.
In the argument in \cite{BF}, the hyperplane symmetries (reflections) are essential.
For instance, for an $1$-unconditional body $K$, the symmetry yields that the polar body $K^\circ$ is also $1$-unconditional.

In general, $K \in \mathcal{K}^3(G)$ has a fundamental domain $\hat{K}$ with respect to the discrete group action by $G$.
Then $\hat{K}$ is a convex truncated cone.
If $G \subset O(n)$ is not a Coxeter group, then the corresponding region $\hat{K}^\circ$ for the polar $K^\circ$ of $K$ is not necessarily convex.
Nevertheless, by using the inequality \cite{IS}*{Proposition 3.2}, in some cases it is possible to obtain the sharp estimate of $\mathcal{P}(K)$ even for convex bodies $K$ without enough hyperplane symmetries.
In this paper we focus on the three dimensional case.

This paper is organized as follows.
In Section \ref{sec:2}, we review necessary facts about two dimensional bodies.
Though all facts here may be well-known, we give short proofs for the self-containedness.
In Section \ref{sec:3}, we give a detailed exposition of the ``signed volume estimate'' for three dimensional $G$-invariant convex bodies.
The principal estimate is the inequality in Lemma \ref{lem:7}, which is frequently used in the arguments in Section \ref{sec:4}.
In Section \ref{sec:4}, we prove the main result (Theorem \ref{thm:2}).
We apply the estimations arranged in Section \ref{sec:3} to the convex bodies in $\mathcal{K}^3(G)$ for each discrete subgroup $G \subset O(3)$.
By means of the inequality in Lemma \ref{lem:7}, the estimation is reduced to the two dimensional result prepared in Sections \ref{sec:2} and \ref{sec:3}.
We also characterize the equality condition for each case in Section \ref{sec:5}.

\section{The two dimensional case}
\label{sec:2}

In this section, we prepare necessary facts about two dimensional bodies.
Although they are essentially proved in \cite{BMMR}, we give proofs of them for the sake of completeness.
The estimate of the volume product $\mathcal{P}(K)$ of a three dimensional convex body $K$ is eventually deduced to that case.
In what follows, we denote by 
\begin{equation*}
\vo, \va, \vb, \vp, \va^\circ, \text{etc.}
\end{equation*}
the position vectors of points $o, A, B, P, A^\circ$, etc.\ in $\R^2$ or $\R^3$, respectively.
And $\va \parallel \vb$ means that the two vectors $\va$ and $\vb$ are parallel.
We denote the {\it positive hull} generated by points
$A_1,\cdots,A_m \in \R^n$ by
\begin{equation*}
\pos (A_1,\dots,A_m)=
\left\{
t_1 \va_1 + \cdots + t_m \va_m \in \R^n; t_1, \dots, t_m \geq 0
\right\}.
\end{equation*}

\subsection{An estimate for two dimensional convex truncated cones}

The following formula estimates the volume product of the intersection of a convex body $K \subset \R^2$ and a positive hull.
\begin{lemma}
\label{lem:1}
Let $K \subset \R^2$ be a convex body. Assume that $A, B \in \partial K$ and $A^\circ, B^\circ \in \partial K^\circ$ with $\va \cdot \va^\circ=\vb \cdot \vb^\circ=1$.
We put $L:= K \cap \pos(A,B)$ and $L^\circ:= K^\circ \cap \pos(A^\circ,B^\circ)$.
Then we have $|L| \, |L^\circ| \geq (\va-\vb)\cdot (\va^\circ-\vb^\circ)/4$.
\end{lemma}

\begin{center}
\includegraphics[height=10em]{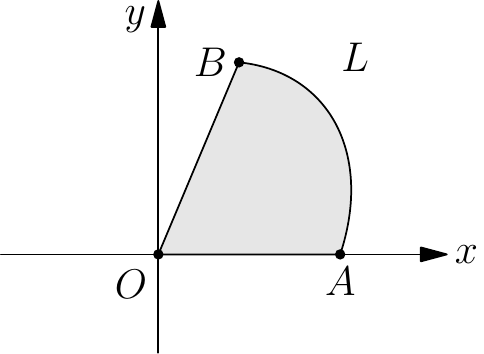}
\includegraphics[height=10em]{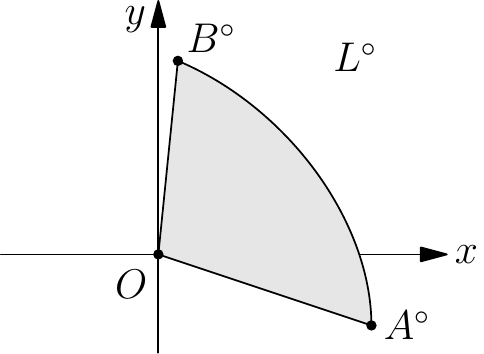}
\end{center}

\begin{proof}
We put
\begin{equation*}
\va=
\begin{pmatrix}
a_1 \\ a_2
\end{pmatrix}, \quad
\vb=
\begin{pmatrix}
b_1 \\ b_2
\end{pmatrix}, \quad
\va^\circ=
\begin{pmatrix}
a^\circ_1 \\ a^\circ_2
\end{pmatrix}, \quad
\vb^\circ=
\begin{pmatrix}
b^\circ_1 \\ b^\circ_2
\end{pmatrix}.
\end{equation*}
Since $\va \cdot \va^\circ = \vb \cdot \vb^\circ=1$, we have
$a_1 a^\circ_1 + a_2 a^\circ_2 =1$ and $b_1 b^\circ_1 + b_2 b^\circ_2 =1$.
For any point $P(x,y)$ in $K$, since the sum of the signed area of the triangle $OAP$ and that of $OPB$ is less than or equal to $|L|$, we have
\begin{equation*}
\frac{1}{2}
\begin{vmatrix}
a_1 & a_2 \\
x & y 
\end{vmatrix}
+
\frac{1}{2}
\begin{vmatrix}
x & y \\
b_1 & b_2
\end{vmatrix}
\leq |L| \text{ for any } 
\begin{pmatrix}
x \\y
\end{pmatrix}
\in K,
\end{equation*}
which means that
\begin{equation*}
\frac{1}{2|L|}
\begin{pmatrix}
 -a_2+b_2 \\ a_1-b_1
\end{pmatrix}
\in K^\circ.
\end{equation*}
Similarly, we obtain
\begin{equation*}
\frac{1}{2|L^\circ|}
\begin{pmatrix}
-a_2^\circ+b_2^\circ \\ a_1^\circ-b_1^\circ 
\end{pmatrix}
 \in K.
\end{equation*}
These test points yields that
\begin{equation*}
4 |L| |L^\circ| 
\geq (a_1-b_1)(a_1^\circ-b_1^\circ) + (a_2-b_2)(a_2^\circ-b_2^\circ)
= (\va-\vb)\cdot(\va^\circ-\vb^\circ).
\end{equation*}
\end{proof}
\begin{remark}
\label{rem:2.2}
We can easily check that the above two test points contained in $L^\circ$ and $L$, respectively.
Lemma \ref{lem:1} is closely related with the second proof of \cite{BMMR}*{Lemma 7}.
\end{remark}

\subsection{The case of cyclically symmetric bodies}

The next lemma is useful for, especially, the case of cyclically symmetric bodies in $\R^2$.

\begin{lemma}
\label{lem:2}
Under the same assumptions in Lemma \ref{lem:1}, assume that $B=R(A)$ and $B^\circ=R(A^\circ)$, where $R$ denotes the rotation of angle $\xi \in (0,\pi)$ around $o$.
Then
\begin{equation}
\label{eq:c}
 |L| \, |L^\circ| \geq \frac{1- \cos \xi}{2}
\end{equation}
holds.
\end{lemma}

\begin{proof}
We may assume that
\begin{equation*}
\begin{aligned}
\va&=
\begin{pmatrix}
 a_1 \\ a_2
\end{pmatrix}
=
\begin{pmatrix}
1 \\ 0
\end{pmatrix}, &
\vb&=
\begin{pmatrix}
 b_1 \\ b_2
\end{pmatrix}
=
\begin{pmatrix}
 \cos \xi \\ \sin \xi
\end{pmatrix}, \\
 \va^\circ &= 
\begin{pmatrix}
 a_1^\circ \\ a_2^\circ
\end{pmatrix}
=
\begin{pmatrix}
 1 \\ a_2^\circ
\end{pmatrix}, &
\vb^\circ&=
\begin{pmatrix}
b_1^\circ \\ b_2^\circ 
\end{pmatrix}
=
\begin{pmatrix}
 \cos \xi-a_2^\circ \sin \xi \\ \sin \xi + a_2^\circ \cos \xi
\end{pmatrix}
\end{aligned}
\end{equation*}
Then we obtain
$4 |L| |L^\circ| \geq 
(1-\cos \xi)(1- \cos \xi + a_2^\circ \sin \xi)
- \sin \xi (a_2^\circ-\sin \xi - a_2^\circ \cos \xi)
=2(1-\cos \xi)$.
\end{proof}

The following is the same equality condition as in the result \cite{BF}*{Corollary 3} or \cite{BMMR}*{Lemma 7}.
\begin{lemma}
\label{lem:3}
Under the same assumptions in Lemma \ref{lem:2}, if $\va \parallel \va^\circ$ and the equality of \eqref{eq:c} holds, then either {\upshape (i)} or {\upshape (ii)} below holds.
\begin{enumerate}[\upshape (i)]
 \item 
$L=\conv \left\{o, A, C, B\right\}$ and
$L^\circ=\conv \left\{o, A^\circ, B^\circ\right\}$, where $\vc=(\va+\vb)/(1+\cos \xi)$.
 \item 
$L=\conv \left\{o, A, B\right\}$ and
$L^\circ=\conv \left\{o, A^\circ, C^\circ, B^\circ\right\}$, where $\vc^\circ=(\va^\circ+\vb^\circ)/(1+\cos \xi)$. 
\end{enumerate}
\end{lemma}

\begin{proof}
Under the setting of Lemma \ref{lem:2}, the assumption $\va \parallel \va^\circ$ implies that $a^\circ_2=0$.
Hence, we can take the test points in the proof of Lemma \ref{lem:1} as
\begin{equation*}
\begin{aligned}
\vc^\circ&:=
\frac{1}{2|L|}
\begin{pmatrix}
 -a_2+b_2 \\ a_1-b_1
\end{pmatrix}
=
\frac{1}{2|L|}
\begin{pmatrix}
\sin \xi \\ 1-\cos \xi
\end{pmatrix}
\in L^\circ, \\
\vc&:=
\frac{1}{2|L^\circ|}
\begin{pmatrix}
-a_2^\circ+b_2^\circ \\ a_1^\circ-b_1^\circ 
\end{pmatrix}
=
\frac{1}{2|L^\circ|}
\begin{pmatrix}
\sin \xi \\
1-\cos \xi
\end{pmatrix}
\in L
\end{aligned}
\end{equation*}
(see Remark \ref{rem:2.2}).
Then we have
\begin{equation*}
\begin{aligned}
|L| &\geq \text{the area of the quadrilateral $oACB$} = \frac{1}{2|L^\circ|}(1-\cos \xi), \\
|L^\circ| &\geq \text{the area of the quadrilateral $o A^\circ C^\circ B^\circ$} = \frac{1}{2|L|}(1-\cos \xi).
\end{aligned}
\end{equation*}
Hence, the equality holds if and only if $L$ and $L^\circ$ coincide with the quadrilaterals $oACB$ and $o A^\circ C^\circ B^\circ$, respectively.
Here note that there exists a constant $\alpha$ such that
\begin{equation*}
\vc=
\frac{1}{2|L^\circ|}
\begin{pmatrix}
\sin \xi \\
1-\cos \xi
\end{pmatrix}
= \frac{\sin (\xi/2)}{|L^\circ|}
\begin{pmatrix}
\cos (\xi/2) \\ \sin(\xi/2)
\end{pmatrix}
=\alpha \frac{\va+\vb}{2}.
\end{equation*}
Since the quadrilateral $oACB$ is convex, we have $\alpha \geq 1$.
Similarly, we obtain
\begin{equation*}
\vc^\circ = \alpha^\circ \frac{\va^\circ+\vb^\circ}{2}, \quad \alpha^\circ \geq 1.
\end{equation*}

\paragraph{Case $\alpha=1$;}
In this case, $L$ coincides with the triangle $oAB$ and the dual face of the edge $AB$ is the vertex $C^\circ$.
That is, $\va \cdot \vc^\circ =1 (=\vb \cdot \vc^\circ)$ holds and
\begin{equation*}
\alpha^\circ= \frac{2}{1+ \cos \xi}
\end{equation*}
Therefore, we obtain
\begin{equation*}
 \vc^\circ = \frac{2}{1+ \cos \xi} \frac{\va^\circ + \vb^\circ}{2} = \frac{\va^\circ + \vb^\circ}{1+\cos \xi},
\end{equation*}
which means that the condition (ii) holds.

\paragraph{Case $\alpha>1$;}
Since $C$ is a vertex of $L$, its dual face is an edge of $L^\circ$.
The edge contains the segment $A^\circ C^\circ$ or $C^\circ B^\circ$.
From $\vc \cdot \vc^\circ=1$, we have $\vc \cdot \va^\circ=1$ or $\vc \cdot \vb^\circ=1$.
It follows that
\begin{equation*}
 \alpha = \frac{2}{1+\cos \xi}, \quad \vc=\frac{\va+\vb}{1+\cos \xi}.
\end{equation*}
Then we have
\begin{equation*}
1= \vc \cdot \vc^\circ=
\frac{\alpha \alpha^\circ}{2} (1+\cos \xi) = \alpha^\circ.
\end{equation*}
Hence, $L^\circ$ coincides with the triangle $o A^\circ B^\circ$, that is, the condition (i) holds.
\end{proof}

\section{Preliminaries for the three dimensional case}
\label{sec:3}

In this section, we recall the method of ``signed volume estimate'' introduced in \cite{IS}.
The exposition here is simpler than that of \cite{IS} and applicable to various truncated cones.
Although the method can be extended to the higher dimensional case, from now on, we concentrate on the three dimensional case.

\subsection{Signed volume of the cone of a ruled surface}
\label{sec:3.1}

Given a convex body $K \in \mathcal{K}^3$ whose interior contains the origin $o$, for any $g \in O(3) \subset GL(3,\R)$, we have
\begin{equation*}
 (g K)^\circ = ({}^t\!g)^{-1} K^\circ = g K^\circ.
\end{equation*}
Hence, $K \in \mathcal{K}^3(G)$ implies that  $K^\circ \in \mathcal{K}^3(G)$ for each subgroup $G \subset O(3)$.

Let $\mathcal C$ be an oriented piecewise $C^1$-curve in $\R^3$ and $\vr(t)$ $(0 \leq t \leq 1)$ a parametrization of $\mathcal C$.
Then we define a vector $\overline{\mathcal{C}} \in \R^3$ by
\begin{equation*}
\overline{\mathcal{C}} := \frac{1}{2} \int_\mathcal{C} \vr \times d \vr = \frac{1}{2} \int_0^1 \vr(t) \times \vr'(t) \,dt,
\end{equation*}
which is independent of the choice of a parametrization of $\mathcal C$.
If the curve $\mathcal C$ is on a plane in $\R^3$ passing through the origin $o$, then $\overline{\mathcal C}$ is a normal vector of the plane.
Let us consider the ruled surface
$o*{\mathcal C}:= \{ \lambda \vu \in \R^3; \vu \in {\mathcal C}, 0 \leq \lambda \leq 1 \}.$
For any $\vx \in \R^3$, the {\it signed volume} of the solid
\begin{equation*}
\vx*(o*{\mathcal C})
:=\{ (1-\nu)\vx+\nu\vxi \in \R^3; \vxi \in o*{\mathcal C}, 0 \leq \nu \leq 1 \}
\end{equation*}
is defined by
\begin{equation*}
 \frac{1}{3} \vx \cdot \overline{\mathcal C}
= \frac{1}{6} \int_0^1 \det(\vx \ \vr(t) \ \vr'(t)) \,dt.
\end{equation*}
Note that this quantity is essential for the signed volume estimate of truncated cones explained later (see Lemma \ref{lem:4}).

Next, we examine the behavior of $\overline{\mathcal C}$ under the action of $O(3)$.
For a rotation $g \in SO(3)$, we have
\begin{equation*}
\overline{g \mathcal{C}} = g \overline{\mathcal C}.
\end{equation*}
In general, this formula does not necessarily hold for $g \in O(3)$.
However, we obtain
\begin{equation*}
\overline{V\mathcal{C}} = - V \overline{\mathcal C}, \quad
\overline{H\mathcal{C}} = - H \overline{\mathcal C},
\end{equation*}
where $V$ and $H$ are the elements of $O(3)$ defined in Section \ref{sec:1.3}.
Indeed, for a parametrization $\vr(t)$ of $\mathcal{C}$, since
\begin{equation*}
\overline{V\mathcal{C}}
= \frac{1}{2} \int_0^1 (V \vr(t)) \times (V \vr'(t)) \,dt
= \frac{1}{2} \int_0^1
\begin{pmatrix}
r_1(t) \\ -r_2(t) \\ r_3(t)
\end{pmatrix}
\times
\begin{pmatrix}
r_1'(t) \\ -r_2'(t) \\ r_3'(t)
\end{pmatrix}
\,dt,
\end{equation*}
we get the first equality.
The second one is also easily verified.

We denote by $-\mathcal{C}$ the curve in $\R^3$ with the same image of the curve $\mathcal C$ but with the opposite direction.
(Note that the notation $-\mathcal{C}$ here is different from the one used in \cite{IS}.)
Then we obtain the formula
\begin{equation*}
 \overline{-\mathcal{C}}=-\overline{\mathcal C}.
\end{equation*}

\subsection{Curves on $\partial K$ and their polars}

Let $K \in \mathcal{K}^3$ with $o \in \interior K$.
For any two points $A, B \in \partial K$ with $\va \nparallel \vb$, let us introduce an oriented curve from $A$ to $B$ on the boundary $\partial K$ defined by
\begin{equation*}
\mathcal{C}(A,B)=
\mathcal{C}_K(A,B):= \rho_K((1-t)\va + t\vb) ((1-t)\va+t\vb):  0 \leq t \leq 1,
\end{equation*}
where $\rho_K$ is the {\it radial function} of $K$ defined by
$\rho_K(\vx) := \max\{ \lambda \geq 0; \lambda \vx \in K \}$
for $\vx \in \R^3 \setminus \{ \vo \}$.
For any points $A_1,\dots,A_m \in \partial K$ with $\va_i \nparallel \va_{i+1}$, we denote by $\mathcal{C}_K(A_1,\dots,A_m)$ the oriented curve on $\partial K$ consists of successive oriented curves $\mathcal{C}(A_i,A_{i+1})$, $i=1,\ldots,m-1$;
\begin{equation*}
\mathcal{C}_K(A_1,\dots,A_m):=\mathcal{C}(A_1,A_2) \cup \cdots \cup \mathcal{C}(A_{m-1},A_m).
\end{equation*}
In particular, if $\mathcal{C}_K(A_1,\dots,A_m)$ is a simple closed curve, that is, $A_m=A_1$, then we denote by $\mathcal{S}_K(A_1,\dots,A_{m-1})$ the part of $\partial K$ enclosed by the curve such that the orientation of the part is compatible with that of the curve.

Let $K \in \K^3$ with $o \in \interior K$ and denote by $\mu_K$ its Minkowski gauge.
Note that $\rho_K(\vx)=1/\mu_K(\vx)$.
Assume that $\partial K$ is a $C^2$-hypersurface in $\R^3$.
From now on, we consider the following class of convex bodies.
\begin{equation*}
\checkK^3:=\left\{
K \in \mathcal{K}^3; o \in \interior K, \text{$K$ is strongly convex, $\partial K$ is of class $C^2$}
\right\},
\end{equation*}
where $K$ is said to be {\it strongly convex} if the Hessian matrix $D^2(\mu_K^2/2)(\vx)$ of $C^2$-function $\mu_K^2/2$ on $\R^3$ is positive definite for each $\vx \in \R^3$ with $|\vx|=1$.
For a convex body $K \in \checkK^3$, we define a $C^2$-map 
$\Lambda=\Lambda_K: \partial K \to \partial K^\circ$ by
\begin{equation*}
\Lambda(\vx)=\nabla \mu_K(\vx) \quad (\vx \in \partial K).
\end{equation*}
If $K$ is strongly convex with the boundary $\partial K$ of class $C^2$, then $K^\circ$ is strongly convex with $\partial K^\circ$ of class $C^1$, and the map $\Lambda: \partial K \to \partial K^\circ$ is a $C^1$-diffeomorphism satisfying that $\vx \cdot \Lambda(\vx)=1$ (see \cite{Sc}*{Section 1.7.2} and \cite{IS}*{Section 3.1}).
Note that the curve $\mathcal{C}(A,B)$ is in the plane passing through the three points $o$, $A$, and $B$, but its image $\Lambda(\mathcal{C}(A,B))$ into $\partial K^\circ$ is not necessarily contained in a plane of $\R^3$.

\begin{lemma}
\label{lem:Lambda}
Let $K \in \checkK^3$ be a convex body and $A, B \in \partial K$ be two points with $\va \nparallel \vb$.
Let $H \subset \R^3$ be the plane passing through the three points $o$, $A$, $B$, and $\pi_H$ denotes the orthogonal projection onto $H$.
Then $\pi_H \circ \Lambda_K = \Lambda_{K \cap H}$ holds on $\partial(K \cap H)$.
\end{lemma}

\begin{proof}
It suffices to consider the case that the points $A, B$ are in the $xy$-plane, i.e.,
$H=\{\vx=(x,y,z) \in \R^3; z=0\}$.
By definition, $\Lambda_K = \nabla \mu_K|_{\partial K}$ is a $C^1$-diffeomorphism from $\partial K$ to $\partial K^\circ$, where $\nabla \mu_K = (\partial_x \mu_K, \partial_y \mu_K, \partial_z \mu_K)$.
Setting $L:=K \cap H$, then $L \subset H$ is strongly convex with its boundary of class $C^2$, $o \in \interior L \subset H$, and $\mu_L = \mu_K|_H$.
Since $\nabla \mu_L = (\partial_x(\mu_K|_H), \partial_y(\mu_K|_H)): H \to H$ and $\Lambda_L := \nabla \mu_L|_{\partial L}: \partial L \to \Lambda_L(\partial L)$, we have
\begin{equation*}
\pi_H \circ \Lambda_K|_{\partial L} = (\partial_x \mu_K, \partial_y \mu_K)|_{\partial L} = \Lambda_L
\end{equation*}
on $\partial L$.
\end{proof}
We next turn to the case of $G$-invariant convex bodies in $\checkK^3$.
For a subgroup $G$ of $O(3)$,
let us introduce the class
\begin{equation*}
\checkK^3(G):= \left\{K \in \checkK^3; gK=K \text{ for all } g \in G\right\}.
\end{equation*}
For a convex body $K$ in this class, the map $\Lambda_K$ behaves $G$-equivariantly as follows.

\begin{lemma}
\label{lem:6}
Let $K \in \checkK^3(G)$.
For any $\vx \in \partial K$ and $g \in G$, we obtain
\begin{equation}
\label{eq:16}
g \Lambda_K(\vx)
= \Lambda_K(g \vx)
\end{equation}
\end{lemma}

\begin{proof}
Let $K \in \checkK^3$.
We first suppose that $\vx \in \R^3 \setminus \{0\}$ and $g \in O(3)$.
By the definition of $\mu_K$, we have $\mu_K(\vx)=\mu_{gK}(g\vx)$.
By differentiating it,
\begin{equation*}
\nabla \mu_K(\vx) = {}^t\!g  \nabla \mu_{gK} (g \vx) = 
g^{-1} \nabla \mu_{gK} (g \vx)
\end{equation*}
holds, because $g \in O(3)$.
Here, assume that $K \in \checkK^3(G)$, $\vx \in \partial K$, and $g \in G$.
Then $gK=K$ and $g \vx \in \partial K$ hold.
Since $\Lambda_K = \nabla \mu_K$ on $\partial K$, we obtain the equality \eqref{eq:16}.
\end{proof}

Not every $G$-invariant convex body $K$ is equipped with the map $\Lambda_K$.
However, owing to the following approximation result, which is a special case of \cite{Sc2}*{pp.\,438}, it suffices to consider only the class $\checkK^3(G)$ for the purpose of this paper.

\begin{proposition}[Schneider]
\label{prop:sch}
Let $G$ be a discrete subgroup of $O(3)$.
Let $K \in \K^3(G)$ be a $G$-invariant convex body.
Then, for any $\varepsilon > 0$ there exists a $G$-invariant convex body $K_\epsilon \in \checkK^3(G)$ having the property that $\delta(K,K_\epsilon) < \varepsilon$, where $\delta$ denotes the Hausdorff distance on $\K^3$.
\end{proposition}
In Section \ref{sec:4}, we frequently use Proposition \ref{prop:sch} for various discrete subgroups of $O(3)$.
\begin{remark}
 Note that Lemmas \ref{lem:Lambda} and \ref{lem:6}, and Proposition \ref{prop:sch} hold for $n$-dimensional case by the same argument.
\end{remark}

\subsection{The signed area estimate}

Let $K \in \checkK^3$.
If we apply the method of signed volume estimate to $K$ (see Section \ref{sec:3.4}), then the lower bound estimate of the volume product $\mathcal{P}(K)$ is reduced to the estimation of some two dimensional situation.
Here we prepare such an estimate.

\begin{proposition}
\label{prop:1}
Let $K \in \checkK^3$.
For any two points $A, B \in \partial K$ with $\va \nparallel \vb$,
it holds that
\begin{equation*}
 \overline{\mathcal{C}(A,B)} \cdot \overline{\Lambda(\mathcal{C}(A,B))} \geq (\va-\vb) \cdot \frac{\Lambda(\va) - \Lambda(\vb)}{4}.
\end{equation*}
Moreover, if $|\va|=|\vb|$, $\va \parallel \Lambda(\va)$, and $\vb \parallel \Lambda(\vb)$ hold, then we obtain
\begin{equation*}
 \overline{\mathcal{C}(A,B)} \cdot \overline{\Lambda(\mathcal{C}(A,B))} \geq \frac{1}{2} (1-\cos \xi),
\end{equation*}
where $\xi$ is the angle between $\va$ and $\vb$.
\end{proposition}

\begin{proof}
Let $H \subset \R^3$ be the plane that contains the curve $\mathcal{C}(A,B)$ and the origin $o$.
By Lemma \ref{lem:Lambda}, we have
\begin{equation*}
\begin{aligned}
o*\pi_H(\Lambda_K(\mathcal{C}(A,B)))
&= o*\Lambda_{K \cap H}(\mathcal{C}(A,B)) \\
&= (K \cap H)^{o_H} \cap \pos(\Lambda_{K \cap H}(A), \Lambda_{K \cap H}(B)),
\end{aligned}
\end{equation*}
where $(K \cap H)^{o_H} := \left\{ y \in H; y \cdot x \leq 1 \text{ for any } x \in K \cap H \right\}$.
Putting $L := (K \cap H) \cap \pos(A, B)$, the convex set $o*\pi_H(\Lambda_K(\mathcal{C}(A,B)))$ is the corresponding $L^\circ$ in Lemma \ref{lem:1}.
Thus,
\begin{equation*}
 |L| \, |L^\circ| \geq (\va-\vb) \cdot \frac{\Lambda(\va) - \Lambda(\vb)}{4}
\end{equation*}
holds.
Moreover, if $\va \parallel \Lambda(\va)$ and $\vb \parallel \Lambda(\vb)$ hold, then we further get $o*\pi_H(\Lambda_K(\mathcal{C}(A,B))) = (K \cap H)^{o_H} \cap \pos(A, B)$.
In this case, by means of Lemma \ref{lem:2}, we can replace the right-hand side of the above inequality by $(1-\cos \xi)/2$.

Finally, we have to check that $\overline{\mathcal{C}(A,B)} \cdot \overline{\Lambda(\mathcal{C}(A,B))} = |L| \, |L^\circ|$.
Since $K$ has its boundary of class $C^2$, then the plane curve $\mathcal{C}(A,B)$ on $\partial K$ is of class $C^2$ and, by definition, we can represent the vector $\overline{\mathcal{C}(A,B)} \in \R^3$ as
\begin{equation}
\label{eq:d}
\overline{\mathcal{C}(A,B)} = |o*\mathcal{C}(A,B)| \vn, \quad
\vn := \frac{\va \times \vb}{|\va \times \vb|}.
\end{equation}
Note that $o*\mathcal{C}(A,B) = K \cap \pos(A, B) = L$ and $\vn$ is the unit normal vector of the plane $H$.
On the other hand, let $\vr(t)$ be a parametrization of the $C^1$-curve $\Lambda(\mathcal{C}(A,B))$ on $\partial K^\circ$ with $\vr(0)=\Lambda(\va)$ and $\vr(1)=\Lambda(\vb)$.
Then we have
\begin{equation*}
\overline{\Lambda(\mathcal{C}(A,B))} \cdot \vn
=
\frac{1}{2} \int_0^1 (\vr(t) \times \vr'(t)) \cdot \vn \,dt.
\end{equation*}
This quantity is nothing but the area of the projection image of the surface $o*\Lambda(\mathcal{C}(A,B))$ to $H$, which is a convex set in $H$.
Consequently,
\begin{equation*}
\overline{\mathcal{C}(A,B)} \cdot \overline{\Lambda(\mathcal{C}(A,B))}
=
|o*\mathcal{C}(A,B)| \vn \cdot \overline{\Lambda(\mathcal{C}(A,B))}
=
|L| \, |o*\pi_H(\Lambda(\mathcal{C}(A,B)))|
=
|L| \, |L^\circ|.
\end{equation*}
\end{proof}

Furthermore, from Lemma \ref{lem:3} we obtain the following
\begin{proposition}
\label{prop:2}
Let $K \in \mathcal{K}^3$.
Let $\va, \vb \in \partial K$ with $\va \nparallel \vb$ and $|\va|=|\vb|$.
Assume that $\va^\circ, \vb^\circ \in \partial K^\circ$ satisfies
$\va \cdot \va^\circ = \vb \cdot \vb^\circ =1$,
$\va \parallel \va^\circ$, and $\vb \parallel \vb^\circ$.
Let $H \subset \R^3$ be the plane passing through the three points $o$, $A$, $B$, and $\pi_H$ denotes the orthogonal projection onto $H$.
Then, setting
\begin{equation*}
L:=K \cap \pos(A,B), \quad
L^\circ:=\pi_H(K^\circ) \cap \pos(A,B),
\end{equation*}
we have
\begin{equation*}
|L| \, |L^\circ| \geq \frac{1}{2} (1-\cos \xi),
\end{equation*}
where $\xi$ is the angle between $\va$ and $\vb$.
The equality holds if and only if either the following {\upshape (i)} or {\upshape (ii)} is satisfied$:$
\begin{enumerate}[\upshape (i)]
 \item 
$L=\conv \left\{o, A, C, B\right\}$ and
$L^\circ= \conv \left\{o, A^\circ, B^\circ\right\}$, where $\vc=(\va+\vb)/(1+\cos \xi)$.
 \item 
$L=\conv \left\{o, A, B\right\}$ and
$L^\circ= \conv \left\{o, A^\circ, C^\circ, B^\circ\right\}$, where $\vc^\circ=(\va^\circ+\vb^\circ)/(1+\cos \xi)$. 
\end{enumerate}
\end{proposition}

\begin{remark}
In the setting of Proposition \ref{prop:2}, $\pos (A,B)=\pos(A^\circ, B^\circ)$ holds.
\end{remark}

\subsection{The signed volume estimate}
\label{sec:3.4}

Now we are in position to state the ``signed volume estimate'', which is a natural generalization of the standard volume estimate used in \cite{Me1986}*{I.2.\,Th\'eor\`eme} and \cite{BF}*{Lemma 11}.

\begin{lemma}[\cite{IS}*{Proposition 3.2}]
\label{lem:4}
Let $K \in \checkK^3$.
Let $\mathcal C$ be a piecewise $C^1$, oriented, simple closed curve on $\partial K$.
Let $\mathcal{S}_K(\mathcal{C}) \subset \partial K$ be a piece of surface enclosed by the curve $\mathcal C$ such that the orientation of the surface is compatible with that of $\mathcal C$.
Then for any point $\vx \in K$, the inequality
\begin{equation*}
\frac{\vx \cdot \overline{\mathcal C}}{3}  \leq |o*\mathcal{S}_K(\mathcal{C})|
\end{equation*} 
holds.
\end{lemma}

\begin{remark}
Although in \cite{IS}*{Proposition 3.2} the boundary $\partial K$ is assumed to be $C^\infty$, the proof works without any changes under the assumptions that $K \in \checkK^3$ and the curve $\mathcal C$ is piecewise $C^1$.
\end{remark}

\begin{lemma}
\label{lem:7}
Under the same assumptions as Lemma \ref{lem:4}, the following inequality holds:
\begin{equation*}
|o*\mathcal{S}_K(\mathcal{C})| \, |o*\mathcal{S}_{K^\circ}(\Lambda(\mathcal{C}))|
\geq \frac{1}{9} \overline{\mathcal{C}} \cdot \overline{\Lambda(\mathcal{C})}.
\end{equation*} 
\end{lemma}

\begin{proof}
By Lemma \ref{lem:4}, we have
$\vx \cdot \overline{\mathcal{C}}/3 \leq |o*\mathcal{S}_K(\mathcal{C})|$ for any $\vx \in K$.
Hence we get a test vector $\overline{\mathcal{C}}/(3 |o*\mathcal{S}_K(\mathcal{C})|) \in K^\circ$.
By the assumption, the $C^1$-map $\Lambda: \partial K \to \partial K^\circ$ can be defined.
By applying Lemma \ref{lem:4} to a piece of surface
$\mathcal{S}_{K^\circ}(\Lambda(\mathcal{C})) \subset \partial K^\circ$, we obtain
\begin{equation*}
 \frac{\overline{\mathcal{C}}}{3 |o*\mathcal{S}_K(\mathcal{C})|} \cdot 
\frac{\overline{\Lambda(\mathcal{C})}}{3} \leq |o*\mathcal{S}_{K^\circ}(\Lambda(\mathcal{C}))|,
\end{equation*}
as claimed.
\end{proof}

In Section \ref{sec:3.1}, we defined the vector $\overline{\mathcal{C}} \in \R^3$ for a piecewise $C^1$-curve $\mathcal{C} \subset \R^3$ by means of line integral.
Let $K \in \mathcal{K}^3$ be a convex body.
Let $A, B \in \partial K$ with $\va \nparallel \vb$.
As we explained in the proof of Proposition \ref{prop:1}, if $\partial K$ is of class $C^1$, then the vector $\overline{\mathcal{C}(A,B)}$ is given by the formula \eqref{eq:d}.
Note that the right-hand side of \eqref{eq:d} is defined whenever the curve $\mathcal{C}(A,B)$ is at least continuous.
Hence, we define the vector $\overline{\mathcal{C}(A,B)} \in \R^3$ by the formula \eqref{eq:d} for such a non-smooth case.

Then we have the following fact, which is necessary for determining the equality conditions of Theorem \ref{thm:2} (see Section \ref{sec:5}).

\begin{lemma}[cf. \cite{IS}*{Lemma 6.2}]
\label{lem:5}
Let $K \in \mathcal{K}^3$ and $A_1, A_2, A_3 \in \partial K$.
Assume that $\mathcal{C}(A_1,A_2,A_3,A_1)$ is a simple closed curve on $\partial K$, that is,
$\mathcal{S}_K(A_1, A_2, A_3)$ is a ``triangle'' on $\partial K$.
Then, for any $\vx \in K$ we have 
\begin{equation*}
 \frac{\vx \cdot \left(
\overline{\mathcal{C}(A_1, A_2)}+\overline{\mathcal{C}(A_2, A_3)}+\overline{\mathcal{C}(A_3, A_1)}
\right)}{3} \leq |o*\mathcal{S}_K(A_1, A_2, A_3)|
\end{equation*}
with equality if and only if $o*\mathcal{S}_K(A_1, A_2, A_3) =
\vx_0 * (o*\mathcal{C}(A_1, A_2, A_3, A_1))$
for some $\vx_0 \in \mathcal{S}_K(A_1, A_2, A_3)$.
\end{lemma}

\begin{proof}
For the inequality part, see \cite{IS}*{Lemma 6.2}.

Assume that the equality holds.
By the construction, if $\vx \notin o*\mathcal{S}_K(A_1, A_2, A_3)$, then the equality does not hold.
Hence, $\vx \in o*\mathcal{S}_K(A_1, A_2, A_3)$.
Then the left-hand side is nothing but the volume of the cone over $o*\mathcal{C}(A_1, A_2, A_3, A_1)$ with the vertex $\vx$.
Since the cone is contained in $o*\mathcal{S}_K(A_1, A_2, A_3)$, the assumption implies that the two solids coincide and $\vx \in \mathcal{S}_K(A_1, A_2, A_3)$.
The converse is obvious.
\end{proof}

\section{Proof of Theorem \ref{thm:2}: inequality}
\label{sec:4}

In this section, we prove the inequalities in Theorem \ref{thm:2} by case analysis.
We start with the case $G=T$.

\subsection{The case $G=T$}
\label{sec:4.1}

\begin{proposition}
\label{prop:5}
The inequality $\mathcal{P}(K) \geq \mathcal{P}(\triangle)$ holds for any $T$-invariant convex body $K \in \K^3(T)$.
\end{proposition}

\begin{proof}
We put
\begin{equation*}
 \va=\frac{1}{\sqrt{3}}
\begin{pmatrix}
1 \\ 1 \\ 1
\end{pmatrix}, \quad
 \vb=\frac{1}{\sqrt{3}}
\begin{pmatrix}
1 \\ -1 \\ -1
\end{pmatrix}, \quad
 \vc=\frac{1}{\sqrt{3}}
\begin{pmatrix}
-1 \\ 1 \\ -1
\end{pmatrix}, \quad
 \vd=\frac{1}{\sqrt{3}}
\begin{pmatrix}
-1 \\ -1 \\ 1
\end{pmatrix}.
\end{equation*}
\begin{center}
\includegraphics[height=10em]{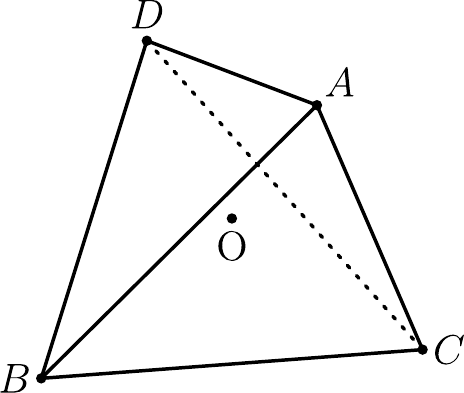}
\end{center}
Then the regular tetrahedron $\triangle$ is represented as $\triangle=\conv\{ A,B,C,D \}$, and we can easily compute that $\mathcal{P}(\triangle)=64/9$.
We denote by $R_A$, $R_B$, $R_C$, and $R_D$ the rotations through the angle $2\pi/3$ about the axes $oA$, $oB$, $oC$, and $oD$, respectively.
That is,
\begin{equation*}
R_A=
\begin{pmatrix}
 0 & 0 & 1 \\
 1 & 0 & 0 \\
 0 & 1 & 0 \\
\end{pmatrix}, \ 
R_B=
\begin{pmatrix}
 0 & 0 & -1 \\
 -1 & 0 & 0 \\
 0 & 1 & 0 \\
\end{pmatrix}, \ 
R_C=
\begin{pmatrix}
 0 & 0 & 1 \\
 -1 & 0 & 0 \\
 0 & -1 & 0 \\
\end{pmatrix}, \ 
R_D=
\begin{pmatrix}
 0 & 0 & -1 \\
 1 & 0 & 0 \\
 0 & -1 & 0 \\
\end{pmatrix}.
\end{equation*}
Under this setting, we see that these four elements generate the group $T$:
\begin{equation*}
 T= \braket{R_A, R_B, R_C, R_D} \subset SO(3).
\end{equation*}
Let $K \in \K^3(T)$.
Since the volume product $\mathcal{P}$ is continuous with respect to the Hausdorff distance on $\K^3$, by Proposition \ref{prop:sch}, it suffices to consider the case that $K \in \checkK^3(T)$.
By a dilation of $K$, 
we may assume that $A \in \partial K$.
Then, by the $T$-symmetry, $B,C,D \in \partial K$ also hold.
Let us consider fundamental domains
$\tilde{K}:=o*\mathcal{S}_{K}(A,B,C)$ and 
$\tilde{K}^\circ:=o*\Lambda(\mathcal{S}_{K}(A,B,C))$
of $K$ and $K^\circ$, respectively.
Then, we have
\begin{equation*}
 |K|=4 |\tilde{K}|, \quad
 |K^\circ|=4 |\tilde{K}^\circ|.
\end{equation*}
By applying Lemma \ref{lem:7} to the curve $\mathcal{C}=\mathcal{C}_K(A,B,C,A)$, we obtain
\begin{equation}
\label{eq:3}
\begin{aligned}
\frac{9}{16}
|K| \, |K^\circ| &\geq 9|\tilde{K}| \, |\tilde{K}^\circ| \\
=&
(\overline{\mathcal{C}(A,B)}+\overline{\mathcal{C}(B,C)}+\overline{\mathcal{C}(C,A)})
\cdot
(\overline{\Lambda(\mathcal{C}(A,B))}+\overline{\Lambda(\mathcal{C}(B,C))}+\overline{\Lambda(\mathcal{C}(C,A))}).
\end{aligned}
\end{equation}

Next, let us compute the right-hand side of \eqref{eq:3}.
Since 
\begin{equation}
\label{eq:e}
\mathcal{C}(B,C)=R_D^2(\mathcal{C}(A,B)), \quad
\mathcal{C}(C,A)=R_D(\mathcal{C}(A,B)),
\end{equation}
we have 
\begin{equation*}
\overline{\mathcal{C}(A,B)}+\overline{\mathcal{C}(B,C)}+\overline{\mathcal{C}(C,A)}
=(E+R_D+R_D^2) \overline{\mathcal{C}(A,B)}.
\end{equation*}
Note that $\overline{\mathcal{C}(A,B)} = |\overline{\mathcal{C}(A,B)}| \,{}^t(0, 1/\sqrt{2}, -1/\sqrt{2})$ from \eqref{eq:d}. Hence,
\begin{equation*}
(E+R_D+R_D^2) \overline{\mathcal{C}(A,B)}
=
\frac{|\overline{\mathcal{C}(A,B)}|}{\sqrt{2}}
\begin{pmatrix}
 1 &  1 & -1 \\
 1 &  1 & -1 \\
-1 & -1 &  1
\end{pmatrix}
\begin{pmatrix}
0 \\ 1 \\ -1
\end{pmatrix}
=
\sqrt{2} |\overline{\mathcal{C}(A,B)}|
\begin{pmatrix}
1 \\ 1 \\ -1
\end{pmatrix}.
\end{equation*}
On the other hand, by Lemma \ref{lem:6}, we have
$\Lambda(\mathcal{C}(B,C))=R_D^2(\Lambda(\mathcal{C}(A,B)))$ and
$\Lambda(\mathcal{C}(C,A))=R_D(\Lambda(\mathcal{C}(A,B)))$
from \eqref{eq:e}, so that
\begin{equation*}
\overline{\Lambda(\mathcal{C}(A,B))}+\overline{\Lambda(\mathcal{C}(B,C))}+\overline{\Lambda(\mathcal{C}(C,A))}
=(E+R_D+R_D^2) \overline{\Lambda(\mathcal{C}(A,B))}
\end{equation*}
holds.
Putting ${}^t(x_1, x_2, x_3):=\overline{\Lambda(\mathcal{C}(A,B))}$, we get
\begin{equation*}
\overline{\mathcal{C}(A,B)} \cdot \overline{\Lambda(\mathcal{C}(A,B))}
=
\frac{1}{\sqrt{2}} (x_2-x_3)|\overline{\mathcal{C}(A,B)}|.
\end{equation*}
Since
\begin{equation*}
R_B^2 R_C = 
\begin{pmatrix}
 1 & 0 & 0 \\
 0 & -1 & 0 \\
 0 & 0 & -1 \\
\end{pmatrix}, \quad R_B^2 R_C(\mathcal{C}(A,B)) = \mathcal{C}(B,A),
\end{equation*}
we obtain
\begin{equation*}
-
\begin{pmatrix}
x_1 \\ x_2 \\ x_3
\end{pmatrix}
=
\overline{\Lambda(\mathcal{C}(B,A))}
=
\overline{R_B^2 R_C(\Lambda(\mathcal{C}(A,B)))}
= 
R_B^2 R_C \overline{\Lambda(\mathcal{C}(A,B))}
=
\begin{pmatrix}
x_1 \\ -x_2 \\ -x_3 
\end{pmatrix},
\end{equation*}
which implies that $x_1=0$. Therefore, 
\begin{equation*}
(E+R_D+R_D^2) \overline{\Lambda(\mathcal{C}(A,B))}
=
\begin{pmatrix}
 1 &  1 & -1 \\
 1 &  1 & -1 \\
-1 & -1 &  1
\end{pmatrix}
\begin{pmatrix}
0 \\ x_2 \\ x_3
\end{pmatrix}
=(x_2-x_3)
\begin{pmatrix}
1 \\ 1 \\ -1
\end{pmatrix}
\end{equation*}
holds.
Thus, the right-hand side of \eqref{eq:3} is computed as
\begin{equation*}
\sqrt{2} |\overline{\mathcal{C}(A,B)}|
\begin{pmatrix}
1 \\ 1 \\ -1
\end{pmatrix}
\cdot
(x_2-x_3)
\begin{pmatrix}
1 \\ 1 \\ -1
\end{pmatrix}
=
3\sqrt{2} (x_2-x_3)|\overline{\mathcal{C}(A,B)}|
=
6 \,\overline{\mathcal{C}(A,B)} \cdot \overline{\Lambda(\mathcal{C}(A,B))}.
\end{equation*}

Finally, we estimate the right-hand side of the above equality from below.
Putting ${}^t(y_1, y_2, y_3):=\Lambda(\va)$, we have
\begin{equation*}
\va \cdot \Lambda(\va) = \frac{1}{\sqrt{3}} (y_1+y_2+y_3)=1, \quad
\begin{pmatrix}
y_1 \\ y_2 \\ y_3
\end{pmatrix}
=
\Lambda(\va) = \Lambda(R_A \va)=R_A \Lambda(\va)
=
\begin{pmatrix}
y_3 \\ y_1 \\ y_2
\end{pmatrix},
\end{equation*}
which implies that $\Lambda(\va)=\va$.
Similarly, $\Lambda(\vb)=\vb$ holds. 
By Proposition \ref{prop:1}, we get
\begin{equation*}
\overline{\mathcal{C}(A,B)} \cdot \overline{\Lambda(\mathcal{C}(A,B))}
\geq 
\frac{1}{4}
(\va-\vb)\cdot
(\Lambda(\va)-\Lambda(\vb))
=\frac{2}{3}.
\end{equation*}
Consequently, we obtain
\begin{equation*}
|K| \, |K^\circ| 
\geq 
\frac{16}{9} \,6 \,\overline{\mathcal{C}(A,B)} \cdot \overline{\Lambda(\mathcal{C}(A,B))}
\geq \frac{64}{9},
\end{equation*}
which completes the proof.
\end{proof}

\subsection{The case $G=O$}
\label{sec:4.2}

\begin{proposition}
\label{prop:6}
The inequality $\mathcal{P}(K) \geq \mathcal{P}(\Diamond)$ holds for any $O$-invariant convex body $K \in \K^3(O)$.
\end{proposition}

\begin{proof}
Let $\Diamond$ be the regular octahedron with vertices
$\{\pm \va, \pm \vb, \pm \vc\}$, where
\begin{equation*}
\va:=
\begin{pmatrix}
1 \\ 0 \\ 0
\end{pmatrix}, \quad
\vb:=
\begin{pmatrix}
0 \\ 1 \\ 0
\end{pmatrix}, \quad
\vc:=
\begin{pmatrix}
0 \\ 0 \\ 1
\end{pmatrix}.
\end{equation*}
Then, all the elements of $O_h$ are described as
\begin{equation*}
\begin{aligned}
\begin{pmatrix}
\pm 1 & 0 & 0 \\
0 & \pm 1 & 0 \\
0 & 0 & \pm 1 \\
\end{pmatrix}, \quad
\begin{pmatrix}
\pm 1 & 0 & 0 \\
0 & 0 & \pm 1 \\
0 & \pm 1 & 0 \\
\end{pmatrix}, \quad
\begin{pmatrix}
0 & \pm 1 & 0 \\
\pm 1 & 0 & 0 \\
0 & 0 & \pm 1 \\
\end{pmatrix}, \\
\begin{pmatrix}
0 & \pm 1 & 0 \\
0 & 0 & \pm 1 \\
\pm 1 & 0 & 0 \\
\end{pmatrix}, \quad 
\begin{pmatrix}
0 & 0 & \pm 1 \\
\pm 1 & 0 & 0 \\
0 & \pm 1 & 0 \\
\end{pmatrix},\quad
\begin{pmatrix}
0 & 0 & \pm 1 \\
0 & \pm 1 & 0 \\
\pm 1 & 0 & 0 \\
\end{pmatrix}.
\end{aligned}
\end{equation*}
In particular, $O=\{g \in O_h; \det g=1\}$.
And we see that $\mathcal{P}(\Diamond)=32/3$.
Let $K \in \K^3(O)$.
In order to examine the volume product of $K$, by Proposition \ref{prop:sch}, we may assume that $K \in \checkK^3(O)$.
Putting $\tilde{K}:=o*\mathcal{S}_{K}(A,B,C)$ and $\tilde{K}^\circ:=o*\Lambda(\mathcal{S}_{K}(A,B,C))$, by the $O$-symmetry, we have
\begin{equation*}
 |K| \, |K^\circ| 
= 64 |\tilde{K}| \, |\tilde{K}^\circ|.
\end{equation*}
By Lemma \ref{lem:7}, we obtain
\begin{equation}
\label{eq:10}
9 |\tilde{K}| \, |\tilde{K}^\circ|
\geq 
 \left(
 \overline{\mathcal{C}(A,B)}
+\overline{\mathcal{C}(B,C)}
+\overline{\mathcal{C}(C,A)}
\right) \cdot
\left(
 \overline{\Lambda(\mathcal{C}(A,B))}
+\overline{\Lambda(\mathcal{C}(B,C))}
+\overline{\Lambda(\mathcal{C}(C,A))}
\right).
\end{equation}

We put
\begin{equation*}
R:=
\begin{pmatrix}
0 & 0 & 1 \\ 
1 & 0 & 0 \\ 
0 & 1 & 0 \\ 
\end{pmatrix}, \
R_A:=
\begin{pmatrix}
1 & 0 & 0 \\ 
0 & 0 & -1 \\ 
0 & 1 & 0 \\ 
\end{pmatrix}, \
R_B:=
\begin{pmatrix}
0 & 0 & 1 \\ 
0 & 1 & 0 \\ 
-1 & 0 & 0 \\ 
\end{pmatrix}
\in O,
\end{equation*}
where $R$ is the rotation through the angle $2\pi/3$ about the axis passing through the origin $o$ and the point $(1,1,1)$, and $R_A$ and $R_B$ are rotations through the angle $\pi/2$ about the axes $oA$ and $oB$, respectively.
Since
\begin{equation*}
\mathcal{C}(B,C)=R(\mathcal{C}(A,B)), \quad 
\mathcal{C}(C,A)=R^2(\mathcal{C}(A,B)),
\end{equation*}
the right-hand side of \eqref{eq:10} becomes
\begin{equation*}
(E+R+R^2)
 \overline{\mathcal{C}(A,B)}
\cdot
(E+R+R^2)
 \overline{\Lambda(\mathcal{C}(A,B))}.
\end{equation*}
Note that $\overline{\mathcal{C}(A,B)}=|\overline{\mathcal{C}(A,B)}| \,{}^t(0,0,1)$.
Here we put ${}^t(x_1,x_2,x_3):=\overline{\Lambda(\mathcal{C}(A,B))}$.
Then, the right-hand side of \eqref{eq:10} equals
\begin{equation*}
|\overline{\mathcal{C}(A,B)}|
\begin{pmatrix}
1 \\ 1 \\ 1
\end{pmatrix}
\cdot
\begin{pmatrix}
 x_1 + x_2 + x_3 \\
 x_1 + x_2 + x_3 \\
 x_1 + x_2 + x_3
\end{pmatrix}
=
3 |\overline{\mathcal{C}(A,B)}| (x_1+x_2+x_3).
\end{equation*}
On the other hand, since $R R_A(\mathcal{C}(A,B)) = R(\mathcal{C}(A,C)) = \mathcal{C}(B,A)$,
we obtain
\begin{equation*}
\begin{pmatrix}
x_2 \\ x_1 \\ -x_3
\end{pmatrix}
=
R R_A \overline{\Lambda(\mathcal{C}(A,B))} = -\overline{\Lambda(\mathcal{C}(A,B))}
=
-
\begin{pmatrix}
x_1 \\ x_2 \\ x_3
\end{pmatrix},
\end{equation*}
hence $x_2=-x_1$. Thus, the right-hand side of \eqref{eq:10} equals
\begin{equation*}
3 |\overline{\mathcal{C}(A,B)}| (x_1+x_2+x_3)
=
3 |\overline{\mathcal{C}(A,B)}| x_3
=
3 \overline{\mathcal{C}(A,B)} \cdot \overline{\Lambda(\mathcal{C}(A,B))}.
\end{equation*}

Since $R_A \va=\va$ and $R_B \vb=\vb$ hold, by a similar argument of the case $G=T$,
we obtain $\Lambda(\va)=\va$ and $\Lambda(\vb)=\vb$.
Therefore, Proposition \ref{prop:1} yields that
\begin{equation*}
\overline{\mathcal{C}(A,B)} \cdot \overline{\Lambda(\mathcal{C}(A,B))} \geq \frac{1}{4} (\va-\vb) \cdot (\Lambda(\va)-\Lambda(\vb)) = \frac{1}{2},
\end{equation*}
so that we obtain $|K| \, |K^\circ| \geq 32/3$, as claimed.
\end{proof}

\begin{remark}
In this case, the minimum of $\mathcal{P}$ is the same as the centrally symmetric case, that is, $S_2=\braket{R_{2}H}=\braket{-E} \cong \mathbb{Z}_2$.
However, the generator $-E$ of $S_2$ is not an element of $O$.
This means that Proposition \ref{prop:6} and \cite{IS}*{Theorem 1} are independent results.
\end{remark}

\subsection{The case $G=I$}

\begin{proposition}
\label{prop:7}
The inequality $\mathcal{P}(K) \geq \mathcal{P}(\Pentagon)$ holds for any $I$-invariant convex body $K \in \K^3(I)$.
\end{proposition}

\begin{proof}
Let $\Pentagon$ be the regular icosahedron with the twelve vertices
\begin{equation*}
(0,\pm 1,\pm \phi), \quad (\pm \phi, 0, \pm 1), \quad (\pm 1,\pm \phi,0),
\end{equation*}
where $\phi=(1+\sqrt{5})/2$
(see \cite{Cox}*{pp.\,52--53}).
We put
\begin{equation*}
\va:=
\begin{pmatrix}
0 \\ 1 \\ \phi
\end{pmatrix}, \quad
\vb:=
\begin{pmatrix}
\phi \\ 0 \\ 1 
\end{pmatrix}, \quad
\vc:=
\begin{pmatrix}
1 \\ \phi \\ 0
\end{pmatrix}.
\end{equation*}
Let $R$ be a rotation through the angle $2 \pi/3$ about the axis ${}^t (1,1,1)$.
Then, we have 
\begin{equation*}
R=
\begin{pmatrix}
0 & 0 & 1 \\
1 & 0 & 0 \\
0 & 1 & 0
\end{pmatrix} \in I, \quad
R(A)=B, \quad
R(B)=C, \quad
R(C)=A.
\end{equation*}
By a simple calculation, we see that
$|\Pentagon|=10(3+\sqrt{5})/3$ and $|\Pentagon^\circ|=2(25-11\sqrt{5})$.
Let $K \in \K^3(I)$.
By Proposition \ref{prop:sch}, we may assume that $K \in \check{K}^3(I)$.
By a dilation, we can also assume that $A,B,C \in \partial K$.
Putting $\tilde{K}:=o*\mathcal{S}_{K}(A,B,C), \tilde{K}^\circ:=o*\Lambda(\mathcal{S}_{K}(A,B,C))$,
by the $I$-symmetry, we obtain
\begin{equation*}
|K| \, |K^\circ| = 400 |\tilde{K}| \, |\tilde{K}^\circ|.
\end{equation*}
By Lemma \ref{lem:7}, we get
\begin{equation}
\label{eq:11}
9 |\tilde{K}| \, |\tilde{K}^\circ| \geq 
\left(
\overline{\mathcal{C}(A,B)}
+
\overline{\mathcal{C}(B,C)}
+
\overline{\mathcal{C}(C,A)}
\right)
\cdot
\left(
\overline{\Lambda(\mathcal{C}(A,B))}
+
\overline{\Lambda(\mathcal{C}(B,C))}
+
\overline{\Lambda(\mathcal{C}(C,A))}
\right).
\end{equation}

Since
\begin{equation*}
\mathcal{C}(B,C)=
R(\mathcal{C}(A,B)), \quad
\mathcal{C}(C,A)=
R^2(\mathcal{C}(A,B)),
\end{equation*}
the right-hand side of \eqref{eq:11} equals
\begin{equation*}
(E+R+R^2)
\overline{\mathcal{C}(A,B)}
\cdot
(E+R+R^2)
\overline{\Lambda(\mathcal{C}(A,B))}.
\end{equation*}
Here, by definition, we have $\overline{\mathcal{C}(A,B)} \parallel {}^t (1, \phi^2, -\phi)$.
Let $R_{AB}$ be the rotation through the angle $\pi$ about the axis through $o$ and the midpoint of the segment $AB$.
Then $R_{AB} \in I$ and 
\begin{equation}
\label{eq:6}
R_{AB} \overline{\Lambda(\mathcal{C}(A,B))}
=
\overline{\Lambda(\mathcal{C}(B,A))}
=
-\overline{\Lambda(\mathcal{C}(A,B))}
\end{equation}
hold.
Note that the three vectors
\begin{equation*}
\va \times \vb =
\begin{pmatrix}
1 \\ \phi^2 \\ -\phi
\end{pmatrix}, \quad
\va-\vb=
\begin{pmatrix}
-\phi \\ 1 \\ \phi-1
\end{pmatrix}, \quad
\va+\vb=
\begin{pmatrix}
\phi \\ 1 \\ \phi+1
\end{pmatrix}
\end{equation*}
are orthogonal to each other.
Putting
$\overline{\Lambda(\mathcal{C}(A,B))}=y_1 (\va \times \vb) + y_2 (\va-\vb) + y_3 (\va+\vb)$,
by \eqref{eq:6}, we obtain
\begin{equation*}
- y_1 (\va \times \vb) - y_2 (\va-\vb) + y_3 (\va+\vb)
=
- y_1 (\va \times \vb) - y_2 (\va-\vb) - y_3 (\va+\vb),
\end{equation*}
so that $y_3=0$.
We put $x\, {}^t(1, \phi^2, -\phi):=\overline{\mathcal{C}(A,B)}$ for a nonzero real number $x$.
Then,
\begin{equation*}
\overline{\mathcal{C}(A,B)} \cdot \overline{\Lambda(\mathcal{C}(A,B))}
= x
\begin{pmatrix}
1 \\ \phi^2 \\ -\phi
\end{pmatrix}
\cdot
\{y_1 (\va \times \vb) + y_2 (\va-\vb)\}
=
4x y_1 \phi^2
\end{equation*}
holds and the right-hand side of \eqref{eq:11} becomes
\begin{equation*}
\begin{aligned}
&
x(E+R+R^2)
\begin{pmatrix}
1 \\ \phi^2 \\ -\phi
\end{pmatrix}
\cdot
(E+R+R^2)
\begin{pmatrix}
y_1-\phi y_2 \\ \phi^2 y_1 + y_2 \\ -\phi y_1 + (\phi-1) y_2
\end{pmatrix}
=4x y_1
\begin{pmatrix}
1 \\ 1 \\ 1
\end{pmatrix}
\cdot
\begin{pmatrix}
1 \\ 1 \\ 1
\end{pmatrix}
=12x y_1 \\
&= \frac{3}{\phi^2} \overline{\mathcal{C}(A,B)} \cdot \overline{\Lambda(\mathcal{C}(A,B))}
= \frac{3(3-\sqrt{5})}{2} \overline{\mathcal{C}(A,B)} \cdot \overline{\Lambda(\mathcal{C}(A,B))}.
\end{aligned}
\end{equation*}

We denote by $R_A$ the rotation through the angle $2\pi/5$ about the axis through $o$ and $A$.
Then $R_A \in I$ and $R_A \va=\va$ hold, so that we have $\Lambda(\va) \parallel \va$.
Similarly, $\Lambda(\vb) \parallel \vb$ holds.
Since $|\va|=|\vb|$, by Proposition \ref{prop:1}, we obtain
\begin{equation*}
\overline{\mathcal{C}(A,B)} \cdot \overline{\Lambda(\mathcal{C}(A,B))}
\geq 
\frac{1}{2} \left(1- \frac{\va \cdot \vb}{|\va| \, |\vb|}\right)=
\frac{1}{2} \left(1- \frac{1}{\sqrt{5}}\right).
\end{equation*}
Consequently,
\begin{equation*}
 |K| \, |K^\circ| \geq \frac{80}{3} (5- 2\sqrt{5}).
\end{equation*}
Since the right-hand side equals the volume product of $\Pentagon$, it completes the proof.
\end{proof}

\subsection{The case $G=C_{\ell h}$}

\begin{proposition}
\label{prop:3}
Assume that $\ell \geq 3$.
Then the inequality $\mathcal{P}(K) \geq \mathcal{P}(P_{\ell})$ holds for any $C_{\ell h}$-invariant convex body $K \in \K^3(C_{\ell h})$.
\end{proposition}

\begin{proof}
Let $K \in \K^3(C_{\ell h})$ for $\ell \geq 3$.
To prove the inequality for $\mathcal{P}$, we may assume that $K \in \checkK^3(C_{\ell h})$ by Proposition \ref{prop:sch}.
We put
\begin{equation*}
\vp:=
\begin{pmatrix}
0 \\ 0 \\ 1
\end{pmatrix}, \quad
\va:=
\begin{pmatrix}
1 \\ 0 \\ 0
\end{pmatrix}, \quad
\vb:=R_\ell \va=
\begin{pmatrix}
\cos \xi \\ \sin \xi \\ 0
\end{pmatrix},
\end{equation*}
where $\xi=2\pi/\ell$.
\begin{center}
\includegraphics[height=10em]{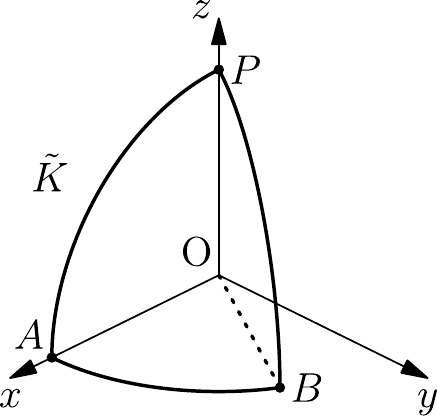}\qquad
\includegraphics[height=10em]{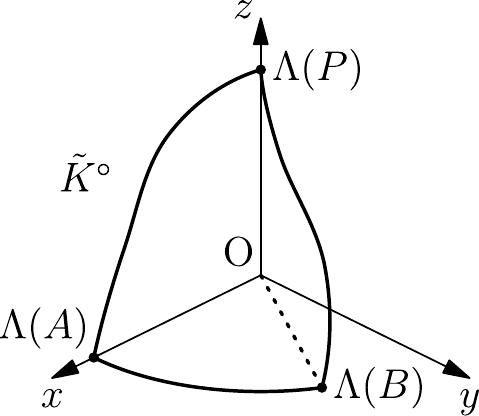}
\end{center}
Recall that $\mathcal{G}$ is a subgroup of $GL(3,\mathbb R)$ defined in Section \ref{sec:1.3}.
Since $gh=hg$ holds for every $g \in \mathcal{G}$ and every $h \in C_{\ell h}$,
we see that $g K \in \check{K}(C_{\ell h})$ for any $g \in \mathcal{G}$.
Thus, we can assume that $P, A \in \partial K$.
Setting $\tilde{K}:=o*\mathcal{S}_{K}(P,A,B)$ and $\tilde{K}^\circ:= o* \Lambda(\mathcal{S}_{K}(P,A,B))$, by the $C_{\ell h}$-symmetries of $K$ and $K^\circ$, we have 
\begin{equation*}
 |K|=2\ell |\tilde{K}|, \quad
 |K^\circ|=2\ell |\tilde{K}^\circ|.
\end{equation*}
By Lemma \ref{lem:7}, we obtain
\begin{equation}
\label{eq:9}
9 |\tilde{K}| \, |\tilde{K}^\circ|
\geq 
 \left(
 \overline{\mathcal{C}(P,A)}
+\overline{\mathcal{C}(A,B)}
+\overline{\mathcal{C}(B,P)}
\right) \cdot
\left(
 \overline{\Lambda(\mathcal{C}(P,A))}
+\overline{\Lambda(\mathcal{C}(A,B))}
+\overline{\Lambda(\mathcal{C}(B,P))}
\right).
\end{equation}
Since $\mathcal{C}(B,P)=R_\ell(\mathcal{C}(A,P))$, we obtain
\begin{equation*}
 \overline{\mathcal{C}(B,P)} 
=R_\ell \overline{\mathcal{C}(A,P)}
=-R_\ell \overline{\mathcal{C}(P,A)}.
\end{equation*}
Since
$\Lambda(\mathcal{C}(B,P))=\Lambda(R_\ell(\mathcal{C}(A,P)))=R_\ell(\Lambda(\mathcal{C}(A,P)))$
by Lemma \ref{lem:6}, we get
\begin{equation*}
 \overline{\Lambda(\mathcal{C}(B,P))} 
=R_\ell \overline{\Lambda(\mathcal{C}(A,P))}
=-R_\ell \overline{\Lambda(\mathcal{C}(P,A))}.
\end{equation*}
It follows from the above equalities that the right-hand side of \eqref{eq:9} equals
\begin{equation*}
\begin{aligned}
& (E-R_\ell) \overline{\mathcal{C}(P,A)} \cdot (E-R_\ell) \overline{\Lambda(\mathcal{C}(P,A))}
+
(E-R_\ell) \overline{\mathcal{C}(P,A)} \cdot \overline{\Lambda(\mathcal{C}(A,B))} \\
& +
\overline{\mathcal{C}(A,B)} \cdot (E-R_\ell) \overline{\Lambda(\mathcal{C}(P,A))}
+
\overline{\mathcal{C}(A,B)} \cdot \overline{\Lambda(\mathcal{C}(A,B))}.
\end{aligned}
\end{equation*}
Here we denote these four terms by $\text{(I)}, \text{(II)}, \text{(III)}$, and $\text{(IV)}$, respectively.

Let us start with the calculation of $\text{(III)}$.
For any $\vx={}^t (x_1, x_2, x_3) \in \R^3$, we have
\begin{equation*}
 (E - R_\ell) \vx =
\begin{pmatrix}
(1-\cos \xi) x_1 + (\sin \xi) x_2 \\
-(\sin \xi) x_1 + (1-\cos \xi) x_2 \\
0
\end{pmatrix}.
\end{equation*}
On the other hand, by definition, $\overline{\mathcal{C}(A,B)} \parallel {}^t (0,0,1)$ holds, which immediately implies that $\text{(III)}=0$.
Next, since
$\Lambda(\mathcal{C}(A,B)) = \Lambda(H(\mathcal{C}(A,B))) =H\Lambda(\mathcal{C}(A,B)))$,
we obtain
\begin{equation*}
\overline{\Lambda(\mathcal{C}(A,B))} = \overline{H(\Lambda(\mathcal{C}(A,B)))} = 
- H \overline{\Lambda(\mathcal{C}(A,B))}.
\end{equation*}
Hence $\overline{\Lambda(\mathcal{C}(A,B))} \parallel {}^t (0,0,1)$ holds,
which means that $\text{(II)}=0$.
Moreover, we can put 
${}^t (0, x_2, 0) := \overline{\mathcal{C}(P,A)}$ and 
$ {}^t (y_1, y_2, y_3) := \overline{\Lambda(\mathcal{C}(P,A))}$ by their definitions.
By the above calculation, we get
\begin{equation*}
\begin{aligned}
\text{(I)}
&=
\begin{pmatrix}
(\sin \xi) x_2 \\ (1- \cos \xi) x_2 \\ 0
\end{pmatrix}\cdot
\begin{pmatrix}
(1-\cos \xi) y_1 + (\sin \xi) y_2 \\
-(\sin \xi) y_1 + (1-\cos \xi) y_2 \\
0
\end{pmatrix}
=2(1-\cos \xi) x_2 y_2 \\
&=2(1-\cos \xi) \overline{\mathcal{C}(P,A)} \cdot \overline{\Lambda(\mathcal{C}(P,A))}.
\end{aligned}
\end{equation*}
Consequently, we obtain
\begin{equation}
\label{eq:4}
|K| \, |K^\circ| \geq \frac{4 \ell^2}{9} \left(
2(1-\cos \xi) \overline{\mathcal{C}(P,A)} \cdot \overline{\Lambda(\mathcal{C}(P,A))}
+ \overline{\mathcal{C}(A,B)} \cdot \overline{\Lambda(\mathcal{C}(A,B))}
\right).
\end{equation}

Finally, we compute the right-hand side of \eqref{eq:4}.
Since $R_\ell \, \vp=\vp$ and $\vp={}^t(0,0,1)$,
we get $R_\ell \Lambda(\vp) = \Lambda(\vp)$ and $\Lambda(\vp) = \vp$.
From $H\va=\va$ and $\vb = R_\ell \va$, we have $H \Lambda(\va) = \Lambda(\va)$ and $\Lambda(\vb) = R_\ell \Lambda(\va)$.
Thus, we can put
\begin{equation*}
\Lambda(\va) =
\begin{pmatrix}
z^\circ_1 \\ z^\circ_2 \\ 0
\end{pmatrix}, \quad
\Lambda(\vb) = 
\begin{pmatrix}
z^\circ_1 \cos \xi - z^\circ_2 \sin \xi \\
z^\circ_1 \sin \xi + z^\circ_2 \cos \xi \\
0
\end{pmatrix}
\end{equation*}
for some $z^\circ_1, z^\circ_2 \in \R$.
Since $z_1^\circ=\va \cdot \Lambda(\va)=1$, Proposition \ref{prop:1} asserts that 
\begin{equation*}
\begin{aligned}
\overline{\mathcal{C}(P,A)} \cdot \overline{\Lambda(\mathcal{C}(P,A))} 
&\geq \frac{1}{4}(\vp-\va)\cdot(\Lambda(\vp)-\Lambda(\va))
=\frac{1}{2}, \\
\overline{\mathcal{C}(A,B)} \cdot \overline{\Lambda(\mathcal{C}(A,B))}
& \geq \frac{1}{4}(\va-\vb)\cdot(\Lambda(\va)-\Lambda(\vb))
=\frac{1}{2} (1-\cos \xi),
\end{aligned}
\end{equation*}
so that
\begin{equation*}
|K| \, |K^\circ| \geq \frac{2 \ell^2}{3} \left(
1-\cos \xi
\right).
\end{equation*}
It is easy to check that the right-hand side equals the volume product of $P_\ell$.
\end{proof}

\subsection{The case $G=D_\ell$}

\begin{proposition}
\label{prop:4}
Assume that $\ell \geq 3$.
Then the inequality $\mathcal{P}(K) \geq \mathcal{P}(P_\ell)$ holds for any $D_\ell$-invariant convex body $K \in \K^3(D_\ell)$.
\end{proposition}

\begin{proof}
Let $K \in \K^3(D_\ell)$ for $\ell \geq 3$.
Recall that $D_\ell=\braket{R_\ell, VH}$.
To prove the inequality, we may assume that $K \in \checkK^3(D_\ell)$ by Proposition \ref{prop:sch}.
We put
\begin{equation*}
\vp=
\begin{pmatrix}
0 \\ 0 \\ 1 
\end{pmatrix}, \quad 
\va=
\begin{pmatrix}
1 \\ 0 \\ 0 
\end{pmatrix}, \quad 
\vq=
\begin{pmatrix}
0 \\ 0 \\ -1 
\end{pmatrix}, \quad 
\vb=
\begin{pmatrix}
\cos \xi \\ \sin \xi \\ 0
\end{pmatrix}.
\end{equation*}
In a similar way as the proof of Proposition \ref{prop:3},
by a linear transformation in $\mathcal{G}'$,
we can assume that $A, B, P, Q \in \partial K$.
Let us consider the closed regions $\tilde{K}:=o*\mathcal{S}_{K}(P,A,B)$ and $\tilde{K}^\circ:=o*\Lambda(\mathcal{S}_{K}(P,A,B))$.
Since $(VH)R_\ell^{-1}(\mathcal{S}_{K}(P,A,B))=\mathcal{S}_{K}(Q,B,A)$ holds, we have
\begin{equation*}
|K| = 2\ell |\tilde{K}|, \quad
|K^\circ| = 2\ell |\tilde{K}^\circ|.
\end{equation*}
Here, we show that
\begin{equation}
\label{eq:1}
(E-R_\ell) \overline{\mathcal{C}(P,A)} \cdot \overline{\Lambda(\mathcal{C}(A,B))} =0.
\end{equation}
For $g \in D_\ell$, we have $g(K)=K$ and $g(\partial K)=\partial K$, so that
\begin{equation*}
g(\mathcal{C}_K (X, Y)) =
\mathcal{C}_{gK} (g(X), g(Y)) =
\mathcal{C}_K (g(X), g(Y))
\end{equation*}
for any $\vx, \vy \in \partial K \text{ with } \vx \nparallel \vy$.
Since $R_\ell, VH \in D_\ell$, the above formula yields that
\begin{equation*}
R_\ell V H(\mathcal{C}(B,A))
=\mathcal{C}(R_\ell V H(B), R_\ell V H(A))
=\mathcal{C}(A,B).
\end{equation*}
Thus, by Lemma \ref{lem:6}, we have
\begin{equation*}
\Lambda(\mathcal{C}(A,B))
= \Lambda(R_\ell V H(\mathcal{C}(B,A)))
= R_\ell V H(\Lambda(\mathcal{C}(B,A))),
\end{equation*}
so that
\begin{equation*}
\overline{\Lambda(\mathcal{C}(A,B))} = 
\overline{R_\ell V H(\Lambda(\mathcal{C}(B,A)))}
=R_\ell V H \overline{\Lambda(\mathcal{C}(B,A))}
=- R_\ell V H \overline{\Lambda(\mathcal{C}(A,B))}.
\end{equation*}
Putting ${}^t (x_1, x_2, x_3) := \overline{\Lambda(\mathcal{C}(A,B))}$, we have
\begin{equation*}
\begin{pmatrix}
x_1 \\ x_2 \\ x_3 
\end{pmatrix}
=
\begin{pmatrix}
 \cos \xi & - \sin \xi & 0 \\
 \sin \xi & \cos \xi & 0 \\
 0 & 0 & 1
\end{pmatrix}
\begin{pmatrix}
-1 & 0 & 0 \\
 0 & 1 & 0 \\
 0 & 0 & 1
\end{pmatrix}
\begin{pmatrix}
 x_1 \\ x_2 \\ x_3
\end{pmatrix}
=
\begin{pmatrix}
-(\cos \xi) x_1 - (\sin \xi) x_2 \\
-(\sin \xi) x_1 + (\cos \xi) x_2 \\
x_3
\end{pmatrix}.
\end{equation*}
Thus
\begin{equation}
\label{eq:2}
\begin{pmatrix}
1+\cos \xi &   \sin \xi \\
  \sin \xi & 1- \cos \xi
\end{pmatrix}
\begin{pmatrix}
 x_1 \\ x_2
\end{pmatrix}
=
\begin{pmatrix}
 0 \\ 0
\end{pmatrix}
\end{equation}
holds.
On the other hand, since
\begin{equation*}
(E-R_\ell) \overline{\mathcal{C}(P,A)} =
\begin{pmatrix}
1-\cos \xi & \sin \xi & 0 \\
-\sin \xi & 1-\cos \xi & 0 \\
0 & 0 & 0
\end{pmatrix}
\begin{pmatrix}
0 \\ |\overline{\mathcal{C}(P,A)}| \\ 0
\end{pmatrix}
=
|\overline{\mathcal{C}(P,A)}|
\begin{pmatrix}
\sin \xi \\ 1-\cos \xi \\ 0
\end{pmatrix},
\end{equation*}
we obtain
\begin{equation*}
(E-R_\ell) \overline{\mathcal{C}(P,A)} \cdot \overline{\Lambda(\mathcal{C}(A,B))}
=
|\overline{\mathcal{C}(P,A)}|
\left(
(\sin \xi) x_1 + (1-\cos \xi)x_2
\right) =0
\end{equation*}
by \eqref{eq:2}, so that \eqref{eq:1} is verified.

Similarly as the case $C_{\ell h}$, by the signed volume estimate, we obtain exactly the same inequality as \eqref{eq:4}.
Since $R_\ell \vp=\vp$ and $\vp \cdot \Lambda(\vp)=1$,
we have $\Lambda(\vp)=R_\ell \Lambda(\vp)$, so that $\Lambda(\vp)=\vp$ holds.
Similarly, $VH \va=\va$ implies that $\Lambda(\va)=\va$.
Hence, $\Lambda(\vb)=\Lambda (R_\ell \va)=R_\ell \Lambda(\va)=R_\ell \va=\vb$ also holds.
Thus, in the same way as in the case $C_{\ell h}$, we can sharply estimate the right-hand side of \eqref{eq:4} from below and obtain the conclusion.
\end{proof}

\subsection{The cases $G=C_{2h}, T_h, S_6, D_{3d}$}

\begin{proof}[Proof of Corollary \ref{cor:1}]
Recall that $\Diamond$ is the regular octahedron, which is $O_h$-invariant.
By \cite{IS}*{Theorem 1}, $\Diamond$ is a minimizer of $\mathcal{P}$ on the set of $S_2$-invariant convex bodies.
Under the matrix representations used in Sections \ref{sec:1.3}, \ref{sec:4.1}, and \ref{sec:4.2},
we have
\begin{equation*}
\begin{aligned}
S_2&=\braket{-E} \subset C_{2h}=\braket{R_2, H}
\subset T_h=\braket{\pm R_A, \pm R_B, \pm R_C, \pm R_D}
\subset O_h, \\
S_2&\subset S_6 \subset D_{3d} \subset O_h' :=g O_h g^{-1},
\end{aligned}
\end{equation*}
where
\begin{equation*}
g:=
\begin{pmatrix}
\frac{\sqrt{2}}{\sqrt{3}} & -\frac{1}{\sqrt{6}} & -\frac{1}{\sqrt{6}} \\
 0 & \frac{1}{\sqrt{2}} & -\frac{1}{\sqrt{2}} \\
 \frac{1}{\sqrt{3}} & \frac{1}{\sqrt{3}} & \frac{1}{\sqrt{3}}
\end{pmatrix}
\in SO(3).
\end{equation*}
It follows from $\Diamond \in \mathcal{K}(O_h)$ that
\begin{equation*}
\mathcal{P}(\Diamond)
=\min_{K \in \mathcal{K}(S_2)} \mathcal{P}(K)
\leq \min_{K \in \mathcal{K}(C_{2h})} \mathcal{P}(K)
\leq \min_{K \in \mathcal{K}(T_h)} \mathcal{P}(K)
\leq \min_{K \in \mathcal{K}(O_h)} \mathcal{P}(K)
\leq
\mathcal{P}(\Diamond).
\end{equation*}
Moreover, since $\Diamond':=g \Diamond$ is also a minimizer of $\mathcal{P}$ on $\mathcal{K}(S_2)$ by \cite{IS}*{Theorem 1} and $\Diamond' \in \mathcal{K}(O_h')$,
we obtain
\begin{equation*}
\mathcal{P}(\Diamond')
=\min_{K \in \mathcal{K}(S_2)} \mathcal{P}(K)
\leq \min_{K \in \mathcal{K}(S_6)} \mathcal{P}(K)
\leq \min_{K \in \mathcal{K}(D_{3h})} \mathcal{P}(K)
\leq \min_{K \in \mathcal{K}(O_h')} \mathcal{P}(K)
\leq
\mathcal{P}(\Diamond').
\end{equation*}
In addition, \cite{IS}*{Theorem 1} asserts that, if $K$ is a minimizer of $\mathcal{P}$ on the set of $S_2$-invariant convex bodies, then $K$ or $K^\circ$ is a parallelepiped.
That is, $K^\circ$ or $K$ is a linear image of $\Diamond$.
Thus, we obtain the conclusion.
\end{proof}

\section{Proof of Theorem \ref{thm:2}: equality conditions}
\label{sec:5}

\subsection{The case $G=T$}

\begin{proposition}
\label{prop:8}
Let $K \in \mathcal{K}^3(T)$. If $\mathcal{P}(K)=\mathcal{P}(\triangle)$ holds, then $K$ is a dilation of $\triangle$ or $\triangle^\circ$.
\end{proposition}

\begin{proof}
Suppose that $K \in \mathcal{K}^3(T)$ satisfies $\mathcal{P}(K)=\mathcal{P}(\triangle)$, i.e., $K$ is a minimizer of $\mathcal{P}$.
By a dilation of $K$, we may assume that $A, B, C, D \in \partial K$, where $A, B, C, D$ are the points used in Section \ref{sec:4.1}.
By the approximation result (Proposition \ref{prop:sch}), there is a sequence $\{K_m\}_{m \in \N} \subset \checkK^3(T)$ such that $K_m \rightarrow K$ $(m \rightarrow \infty)$ in the Hausdorff distance.
In addition, normalizing $K_m$ by their dilations, we can assume $A,B,C,D \in \partial K_m$.
Notice that, even though this normalization is done, the convergence $K_m \rightarrow K$ $(m \rightarrow \infty)$ still holds.
We denote the curve $\mathcal{C}_{K_m}(A,B)$ by $\mathcal{C}_m(A,B)$ for short.
For each convex body $K_m$, we can define the map $\Lambda_m:\partial K_m \to \partial K_m^\circ$.
Let $H$ be the plane through $o$, $A$, and $B$, and $\pi_H$ be the orthogonal projection onto $H$.
We put
\begin{equation*}
 L_m:=o*\mathcal{C}_m(A,B) \subset H, \quad L_m^\circ:=o*\pi_H(\Lambda_m(\mathcal{C}_m(A,B))) \subset H.
\end{equation*}
Then by the argument in the proof of Proposition \ref{prop:5}, we have
\begin{equation}
\label{eq:7}
\frac{9}{16} \mathcal{P}(K_m) \geq 6 |L_m| \, |L_m^\circ| \geq 4.
\end{equation}

Now we examine the two dimensional subsets $L_m$ and $L_m^\circ$ of $H$.
By the $T$-symmetry, we have $A^\circ:=\Lambda_m(A)=A$ and  $B^\circ:=\Lambda_m(B)=B$, and
\begin{equation*}
L_m=K_m \cap \pos(A,B), \quad
L_m^\circ=\pi_H(K_m^\circ) \cap \pos(A,B)
\end{equation*}
hold. Putting
\begin{equation*}
L:=K \cap \pos(A,B), \quad
L^\circ:=\pi_H(K^\circ) \cap \pos(A,B),
\end{equation*}
then we obtain that $L_m \rightarrow L$ and $L_m^\circ \rightarrow L^\circ$ on $H$ in the Hausdorff distance.
Since $K$ is a minimizer of $\mathcal{P}$,
taking $m \rightarrow \infty$ in \eqref{eq:7}, we have
\begin{equation*}
|L| \, |L^\circ|= \frac{2}{3},
\end{equation*}
which means that Proposition \ref{prop:2} holds with equality. 
Thus, either of the following two cases occurs.
\begin{enumerate}[{Case} (i):]
 \item 
$L=\conv \left\{o, A, Z, B\right\}$ and
$L^\circ= \conv \left\{o, A^\circ, B^\circ\right\}$, where $\vz=3(\va+\vb)/2$.
 \item 
$L=\conv \left\{o, A, B\right\}$ and
$L^\circ= \conv \left\{o, A^\circ, Z^\circ, B^\circ\right\}$, where $\vz^\circ=3(\va^\circ+\vb^\circ)/2$. 
\end{enumerate}

First, we consider Case (ii).
\begin{center}
\includegraphics[height=13em]{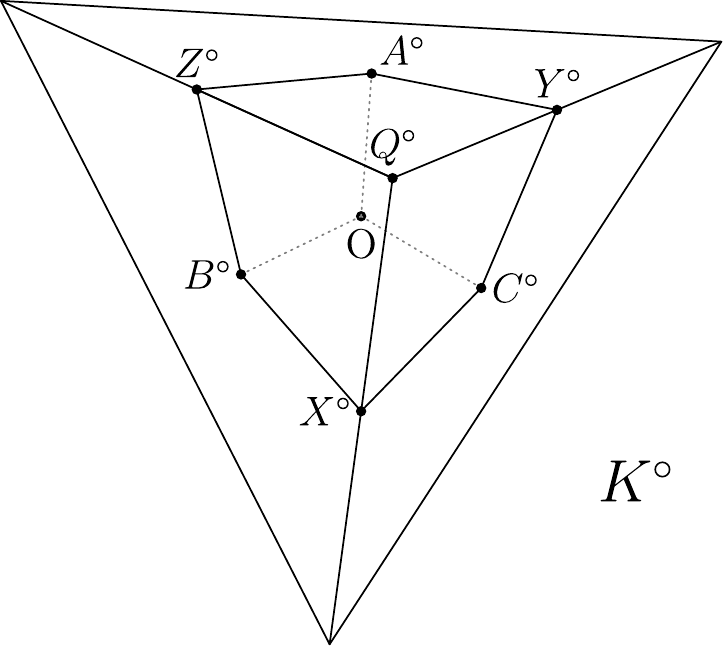}
\end{center}
Since $Z^\circ \in \pi_H(K^\circ)$, there exists $s \in \R$ such that 
$\vz^\circ + s \va \times \vb \in K^\circ$.
We note that $\vz^\circ \parallel (\va^\circ+\vb^\circ)/2=(\va+\vb)/2$.
Let $R$ be the rotation through the angle $\pi$ about the axis through $o$ and the midpoint of $A$ and $B$.
Then, we see that $R \in T$ and
\begin{equation*}
 R(\vz^\circ + s \va \times \vb) = \vz^\circ - s \va \times \vb \in R(K^\circ) = K^\circ.
\end{equation*}
Thus, the convexity of $K^\circ$ implies that $Z^\circ \in K^\circ$. 
Moreover, since $o, A^\circ, B^\circ, Z^\circ \in \pos(A,B)$, we get
\begin{equation*}
L^\circ = \conv \left\{o, A^\circ, Z^\circ, B^\circ\right\} \subset K^\circ \cap \pos(A,B)
\subset \pi_H(K^\circ) \cap \pos(A,B) = L^\circ,
\end{equation*}
so that $L^\circ=K^\circ \cap \pos(A,B)$.
It follows from $A=A^\circ$ and $B=B^\circ$ that
\begin{equation}
\label{eq:8}
L^\circ=o*\mathcal{C}_{K^\circ}(A^\circ, B^\circ), \quad
\overline{\mathcal{C}_{K^\circ}(A^\circ, B^\circ)} = |L^\circ| \frac{\va \times \vb}{|\va \times \vb|}.
\end{equation}
Putting 
$Y^\circ:=R_D Z^\circ$ and $X^\circ:=R_D^2 Z^\circ$, we get
\begin{equation*}
o*\mathcal{C}_{K^\circ}(C^\circ,A^\circ) = \conv\{o, C^\circ, Y^\circ, A^\circ\}, \quad
o*\mathcal{C}_{K^\circ}(B^\circ,C^\circ) = \conv\{o, B^\circ, X^\circ, C^\circ\}.
\end{equation*}
As for $L$, since $L=\conv \{o, A, B\}=K \cap \pos(A,B)$, we have
\begin{equation*}
L=o*\mathcal{C}_K(A,B), \quad
\overline{\mathcal{C}_K(A,B)} = |L| \frac{\va \times \vb}{|\va \times \vb|}.
\end{equation*}

Now we consider the truncated convex cone $o*\mathcal{S}_K(A,B,C)$.
By the $T$-symmetry, we have
\begin{equation*}
\begin{aligned}
&\overline{\mathcal{C}_K(A,B)}+\overline{\mathcal{C}_K(B,C)}+\overline{\mathcal{C}_K(C,A)}
=
(E+R_D+R_D^2)\overline{\mathcal{C}_K(A,B)} \\
&=
|L| (E+R_D+R_D^2) \frac{\va \times \vb}{|\va \times \vb|}
=\sqrt{2} |L|
\begin{pmatrix}
 1 \\ 1 \\ -1
\end{pmatrix}.
\end{aligned}
\end{equation*}
Moreover, Lemma \ref{lem:5} asserts that
\begin{equation*}
\frac{\sqrt{2} |L| \vx}{3} \cdot
\begin{pmatrix}
 1 \\ 1 \\ -1
\end{pmatrix}
=
\frac{\vx}{3}
 \cdot 
(\overline{\mathcal{C}_K(A,B)}+\overline{\mathcal{C}_K(B,C)}+\overline{\mathcal{C}_K(C,A)})
\leq |o* \mathcal{S}_K(A,B,C)| = \frac{|K|}{4}
\end{equation*}
for any $\vx \in K$. Hence,
\begin{equation*}
\vq^\circ:=
\frac{4\sqrt{2} |L|}{3|K|}
\begin{pmatrix}
 1 \\ 1 \\ -1
\end{pmatrix} \in K^\circ.
\end{equation*}
We put $\hat{K}^\circ:=o*\mathcal{S}_{K^\circ}(A^\circ,B^\circ,C^\circ)$, then it is easy to check that $\vq^\circ \in \hat{K}^\circ$.
By \eqref{eq:8},
we can calculate the volume of the solid
$Q^\circ*(o* \mathcal{C}_{K^\circ}(A^\circ,B^\circ,C^\circ,A^\circ)) \subset \hat{K}^\circ$ as 
\begin{equation*}
\begin{aligned}
&
\frac{1}{3} \vq^\circ \cdot \left(
\overline{\mathcal{C}_{K^\circ}(A^\circ,B^\circ)}
+
\overline{\mathcal{C}_{K^\circ}(B^\circ,C^\circ)}
+
\overline{\mathcal{C}_{K^\circ}(C^\circ,A^\circ)}
\right)
=
\frac{1}{3} \vq^\circ \cdot
(E+R_D+R_D^2) |L^\circ| \frac{\va \times \vb}{|\va \times \vb|}
\\
&
=
\frac{4\sqrt{2} |L|}{9|K|}
\begin{pmatrix}
 1 \\ 1 \\ -1
\end{pmatrix}
\cdot
\sqrt{2}{|L^\circ|}
\begin{pmatrix}
 1 \\ 1 \\ -1
\end{pmatrix}
=
\frac{16}{9 |K|} = \frac{|K^\circ|}{4},
\end{aligned}
\end{equation*}
which is exactly the volume of $\hat{K}^\circ$.
Hence, we can apply Lemma \ref{lem:5} with equality to get
$\hat{K}^\circ=Q^\circ*(o* \mathcal{C}_{K^\circ}(A^\circ,B^\circ,C^\circ,A^\circ))$.
By the convexity, the segment $Q^\circ A^\circ$ is on the boundary of $\hat{K}^\circ$ and that of $K^\circ$.
Here, the endpoints of the three vectors
\begin{equation*}
\va^\circ=
\frac{1}{\sqrt{3}}
\begin{pmatrix}
1 \\ 1 \\ 1
\end{pmatrix}, \quad
\vy^\circ=
\frac{1}{\sqrt{3}}
\begin{pmatrix}
0 \\ 3 \\ 0
\end{pmatrix}, \quad
\vz^\circ=
\frac{1}{\sqrt{3}}
\begin{pmatrix}
3 \\ 0 \\ 0
\end{pmatrix}
\end{equation*}
are on the plane defined by $(x+y+z)/\sqrt{3}=1$.
Thus, putting $\vq^\circ=s \, {}^t (1, 1, -1)$ $(s>0)$, we have
$\va^\circ \cdot(\vq^\circ-\va^\circ) \geq 0$, i.e.,
\begin{equation*}
\frac{1}{\sqrt{3}}
\begin{pmatrix}
 1 \\ 1 \\ 1
\end{pmatrix}
\cdot \vq^\circ
=
\frac{s}{\sqrt{3}} \geq 1,
\end{equation*}
so that $s \geq \sqrt{3}$.
On the other hand, $\va \cdot \vq^\circ =s/\sqrt{3} \leq 1$ holds, which yields that $s=\sqrt{3}$ and $\vq^\circ=\sqrt{3} \,{}^t(1, 1, -1)$.
Hence, $\vq^\circ \cdot \va=\vq^\circ \cdot \vb=\vq^\circ \cdot \vc=1$.
By the definition of the polar, we have $K \subset \{\vx \in \R^3; \vx \cdot \vq^\circ \leq 1\}$.
Thus, we obtain
\begin{equation*}
\conv\{o,A,B,C\} \subset K \cap \pos(A,B,C) \subset \{\vx; \vx \cdot \vq^\circ \leq 1\} \cap \pos(A,B,C)
= \conv\{o,A,B,C\},
\end{equation*}
so that $K \cap \pos (A,B,C) = \conv \{o,A,B,C\}$.
By the $T$-symmetry of $K$, we have $K=\conv\{A,B,C,D\}$, which means that $K$ is a dilation of $\triangle$.

Finally, we consider Case (i).
It is clear that $L=o*\mathcal{C}_K(A,B)$.
As for $L^\circ$,
we have $L^\circ=\conv\{o,A^\circ, B^\circ\}$ and $o, A^\circ(=A), B^\circ(=B) \in L^\circ$.
It then follows from the definition of $L^\circ$ that
\begin{equation*}
L^\circ \subset K^\circ \cap \pos(A,B) \subset \pi_H(K^\circ) \cap \pos(A,B) = L^\circ,
\end{equation*}
so that $K^\circ \cap \pos(A,B) = L^\circ$, hence \eqref{eq:8} holds.
Then, similarly as Case (ii), Lemma \ref{lem:5} asserts that
\begin{equation*}
\vq:=
\frac{4\sqrt{2} |L^\circ|}{3|K^\circ|}
\begin{pmatrix}
 1 \\ 1 \\ -1
\end{pmatrix} \in K.
\end{equation*}
After that, by the same argument exchanging the roles of $K$ and $K^\circ$ in Case (ii), we get
$K^\circ=\conv\{A^\circ,B^\circ,C^\circ,D^\circ\}$.
\end{proof}

\subsection{The cases $G=O,I$}

One can show the following similarly as Proposition \ref{prop:8}, thus we omit their proofs.

\begin{proposition}
\label{prop:9}
\begin{enumerate}[\upshape (i)]
 \item 
Let $K \in \mathcal{K}^3(O)$. If $\mathcal{P}(K)=\mathcal{P}(\Diamond)$ holds,
then $K$ is a dilation of $\Diamond$ or $\Diamond^\circ$.
 \item 
Let $K \in \mathcal{K}^3(I)$. If $\mathcal{P}(K)=\mathcal{P}(\Pentagon)$ holds,
then $K$ is a dilation of $\Pentagon$ or $\Pentagon^\circ$.
\end{enumerate}
\end{proposition}

\subsection{The case $G=D_\ell$}

\begin{proposition}
\label{prop:10}
Assume that $\ell \geq 3$.
Let $K \in \mathcal{K}^3(D_\ell)$.
If $\mathcal{P}(K)=\mathcal{P}(P_\ell)$ holds, then $K$ coincides with $P_\ell$ or $P_\ell^\circ$ up to a linear transformation in $\mathcal{G}'$.
\end{proposition}

\begin{proof}
We use the symbols $\vp$, $\va$, $\vq$, and $\vb$ defined in the proof of Proposition \ref{prop:4}.
Without loss of generalities, we may assume that the minimizer $K$ satisfies $P, A, B \in \partial K$.
By Proposition \ref{prop:sch}, there is a sequence $\{K_m\}_{m \in \N} \subset \checkK^3(D_\ell)$ such that $K_m \rightarrow K$ $(m \rightarrow \infty)$.
Normalizing $K_m$ by a linear transformation in $\mathcal{G}'$, we can assume that
$P, A, B \in \partial K_m$ for every $m$.
With this change, the convergence $K_m \rightarrow K$ $(m \rightarrow \infty)$ still holds.
Let $\pi_2$ and $\pi_3$ be the orthogonal projections to $zx$-plane and $xy$-plane, respectively.
Let $H$ be the plane through $o,P,B$, and $\pi_H$ be the orthogonal projection onto $H$.
For each $K_m$, by the same argument as in the proof of Proposition \ref{prop:8}, we have $P^\circ:=\Lambda_m(P)=P$, $A^\circ:=\Lambda_m(A)=A$, and $B^\circ:=\Lambda_m(B)=B$, where $\Lambda_m:\partial K_m \to \partial K_m^\circ$.
Moreover, we put
\begin{equation*}
\begin{aligned}
L_m
&:=o*\mathcal{C}_m(P,A)
=K_m \cap \pos(P,A), \\
L_m^\circ
&:=o*\pi_2(\Lambda_m(\mathcal{C}_m(P,A)))
=\pi_2(K_m^\circ) \cap \pos(P,A), \\
M_m
&:=o*\mathcal{C}_m(A,B)
=K_m \cap \pos(A,B), \\
M_m^\circ
&:=o*\pi_3(\Lambda_m(\mathcal{C}_m(A,B)))
=\pi_3(K_m^\circ) \cap \pos(A,B).
\end{aligned}
\end{equation*}
As $m \rightarrow \infty$, these two dimensional sets converge to
\begin{equation*}
\begin{aligned}
L
&:=K \cap \pos(P,A), \quad
L^\circ
:=\pi_2(K^\circ) \cap \pos(P,A), \\
M
&:=K \cap \pos(A,B), \quad
M^\circ
:=\pi_3(K^\circ) \cap \pos(A,B),
\end{aligned}
\end{equation*}
respectively, in the Hausdorff distance.
Since the inequality \eqref{eq:4} holds for each $K_m$, by taking the limit $m \to \infty$,
we obtain
\begin{equation*}
|K| \, |K^\circ| \geq \frac{4 \ell^2}{9} \left(
2(1-\cos \xi) |L|\,|L^\circ|
+ |M|\,|M^\circ|
\right).
\end{equation*}
Since $K$ is a minimizer of $\mathcal{P}$, by the inequality in Proposition \ref{prop:2}, we get
\begin{equation*}
|L|\,|L^\circ|=\frac{1}{2}, \quad
|M|\,|M^\circ|=\frac{1-\cos \xi}{2},
\end{equation*}
which implies that Proposition \ref{prop:2} holds with equality for $L$ and also for $M$.
Hence, for the two dimensional subsets  $L$ and $L^\circ$ either of the following two cases happens.
\begin{enumerate}[{Case} (A):]
 \item
$L=\conv \left\{o, P, C, A\right\}$ and
$L^\circ= \conv \left\{o, P^\circ, A^\circ\right\}$, where $\vc=\vp+\va$.
 \item 
$L=\conv \left\{o, P, A\right\}$ and
$L^\circ= \conv \left\{o, P^\circ, C^\circ, A^\circ\right\}$, where $\vc^\circ=\vp^\circ+\va^\circ$.
\end{enumerate}
In addition, $M$ and $M^\circ$ are characterized as one of the following.
\begin{enumerate}[{Case} (a):]
 \item
$M=\conv \left\{o, A, D, B\right\}$ and
$M^\circ= \conv \left\{o, A^\circ, B^\circ\right\}$, where $\vd=(\va+\vb)/(1+\cos \xi)$.
 \item 
$M=\conv \left\{o, A, B\right\}$ and
$M^\circ= \conv \left\{o, A^\circ, D^\circ, B^\circ\right\}$, where $\vd^\circ=(\va^\circ+\vb^\circ)/(1+\cos \xi)$. 
\end{enumerate}

Thus, it suffices to consider only four cases.
Before that, we give some remarks.

\begin{remark}
\label{rem:5}
\begin{enumerate}[\upshape (i)]
 \item 
In Case (B), since $C^\circ \in L^\circ \subset \pi_2(K^\circ)$,
there exists $s \in \R$ such that $\vc^\circ + {}^t(0,s,0) \in K^\circ$.
In the following, we put $\vc':=\vc^\circ + {}^t(0,s,0)$.
 \item 
In Case (b), since $D^\circ \in M^\circ \subset \pi_3(K^\circ)$,
there exists $t \in \R$ such that $\vd^\circ + {}^t(0,0,t) \in K^\circ$.
In addition, we have $R_\ell VH(\vd^\circ + {}^t(0,0,t)) \in K^\circ$, so that
\begin{equation*}
\begin{pmatrix}
1 \\ \frac{\sin \xi}{1+\cos \xi} \\ t 
\end{pmatrix},
\begin{pmatrix}
1 \\ \frac{\sin \xi}{1+\cos \xi} \\ -t 
\end{pmatrix} \in K^\circ,
\text{ and }
 \vd^\circ = 
\begin{pmatrix}
1 \\ \frac{\sin \xi}{1+\cos \xi} \\ 0 
\end{pmatrix} \in K^\circ
\end{equation*}
by the convexity of $K^\circ$.
Hence,
\begin{equation*}
M^\circ=\conv \left\{o, A^\circ, D^\circ, B^\circ\right\} \subset K^\circ \cap \pos(A,B) \subset \pi_3(K^\circ) \cap \pos(A,B)= M^\circ
\end{equation*}
holds, so that $M^\circ =\conv \left\{o, A^\circ, D^\circ, B^\circ\right\}=K^\circ \cap \pos(A,B)$.
\end{enumerate}
\end{remark}

\paragraph{Case (A) and (a):} 
Since $P,A \in K^\circ$ and $L^\circ= \conv \left\{o, P^\circ, A^\circ\right\}$, by the convexity of $K^\circ$, we have
\begin{equation}
\label{eq:g}
L^\circ= \conv \left\{o, P^\circ, A^\circ\right\} \subset K^\circ \cap \pos(P,A)
\subset \pi_2(K^\circ) \cap \pos(P,A)= L^\circ,
\end{equation}
that is,
$L^\circ=\pi_2(K^\circ) \cap \pos(P,A)=K^\circ \cap \pos(P,A)$.
Similarly, $M^\circ=K^\circ \cap \pos(A,B)$.
By exchanging the roles of $K$ and $K^\circ$, this case is reduced to Case (B) and (b) below.

\paragraph{Case (A) and (b):}
We put $\hat{K}^\circ:=o*\mathcal{S}_{K^\circ}(P^\circ,A^\circ,B^\circ) = K^\circ \cap \pos(P,A,B)$.
Since $L^\circ=K^\circ \cap \pos(P,A)$ from \eqref{eq:g} and $M^\circ=\conv \left\{o, A^\circ, D^\circ, B^\circ\right\}$ from Remark \ref{rem:5}\,(ii),
$\conv\{o, P^\circ, A^\circ, D^\circ, B^\circ\} \subset \hat{K}^\circ$ holds.
On the other hand, we have $\pi_2(\hat{K}^\circ) \subset L^\circ$
and 
$\pi_H(\hat{K}^\circ) \subset R_\ell(L^\circ)$ for $\ell \geq 4$, which mean that
$\hat{K}^\circ \subset \conv\{o, P^\circ, A^\circ, D^\circ, B^\circ\}$ for $\ell \geq 4$.
However, $\pi_2(\hat{K}^\circ) \subset L^\circ$ and $\pi_H(\hat{K}^\circ) \subset R_\ell(L^\circ)$ do not hold for $\ell=3$.
Therefore, we give another proof of $\hat{K}^\circ=\conv\{o, P^\circ, A^\circ, D^\circ, B^\circ\}$
which is applicable for all $\ell \geq 3$.

Let us also consider $\hat{K}:=o*\mathcal{S}_K(P, A, B) = K \cap \pos(P,A,B)$.
In the same way as $\hat{K}^\circ$, we have $\conv\{o, P, C, A, B, R_{\ell}(C)\} \subset \hat{K}$.
Then, it follows from the $D_\ell$-symmetries of $K$ and $K^\circ$ that
\begin{equation}
\label{eq:f}
2\ell |\conv\{o, P, C, A, B, R_{\ell}(C)\}|\ 2\ell |\conv\{o, P^\circ, A^\circ, D^\circ, B^\circ\}|
\leq |K|\,|K^\circ|=\frac{2\ell^2}{3}(1-\cos\xi).
\end{equation}
On the other hand, by an easy calculation, we have $|\conv\{o, P, C, A, B, R_{\ell}(C)\}|=(\sin\xi)/2$
and $|\conv\{o, P^\circ, A^\circ, D^\circ, B^\circ\}|=(\sin\xi)/3(1+\cos\xi)$.
Hence, the inequality \eqref{eq:f}, in fact, holds with equality.
This means that $\hat{K}^\circ=\conv\{o, P^\circ, A^\circ, D^\circ, B^\circ\}$.

Again, by the $D_\ell$-symmetry, $K^\circ$ is nothing but a regular $\ell$-bipyramid
such that the cross section of $K^\circ$ with the $xy$-plane is the regular $\ell$-gon with vertices $(R_\ell)^j(D^\circ)$ $(j=0,\dots,\ell-1)$.

\paragraph{Case (B) and (a):}
We put $\hat{K}^\circ:=o*\mathcal{S}_{K^\circ}(P^\circ,A^\circ,B^\circ) = K^\circ \cap \pos(P,A,B)$.
Then, in a similar way as the above case,
$\conv\{o, A^\circ, B^\circ, P^\circ, C^\circ, R_\ell(C^\circ)\} \subset \hat{K}^\circ$ holds.
Again, by a similar argument in the above case, we obtain
$\hat{K}^\circ=\conv\{o, A^\circ, B^\circ, P^\circ, C^\circ, R_\ell(C^\circ)\}$.
By the $D_\ell$-symmetry, $K^\circ$ coincides with an $\ell$-regular right prism.

\paragraph{Case (B) and (b):}
We put $\hat{K}^\circ:=o*\mathcal{S}_{K^\circ}(P^\circ, C', A^\circ, D^\circ, B^\circ, R_\ell(C'))$.
For each convex body $K_m \in \checkK^3(D_\ell)$, we have
\begin{equation*}
\begin{aligned}
 \frac{2{\ell}}{3 |K_m|} \left(
 \overline{\mathcal{C}_m(P,A)}
+\overline{\mathcal{C}_m(A,B)}
+\overline{\mathcal{C}_m(B,P)}
\right)
&=
\frac{2{\ell}}{3 |K_m|}
\left(
|L_m|
\begin{pmatrix}
 \sin \xi \\ 1-\cos \xi \\ 0
\end{pmatrix}
+
|M_m|
\begin{pmatrix}
0 \\ 0 \\ 1
\end{pmatrix}
\right)
\in K_m^\circ.
\end{aligned}
\end{equation*}
As $m \rightarrow \infty$, we obtain
\begin{equation*}
\vq^\circ:=
\frac{2{\ell}}{3 |K|}
\left(
|L|
\begin{pmatrix}
 \sin \xi \\ 1-\cos \xi \\ 0
\end{pmatrix}
+
|M|
\begin{pmatrix}
0 \\ 0 \\ 1
\end{pmatrix}
\right)
\in K^\circ.
\end{equation*}
Now, we calculate the (signed) volume of the solid
\begin{equation*}
\tilde{K}^\circ:=Q^\circ*(o*\mathcal{C}_{K^\circ}(P^\circ, C', A^\circ, D^\circ, B^\circ, R_\ell(C'), P^\circ)) \subset K^\circ.
\end{equation*}
We put $L_1^\circ:=o*\mathcal{C}_{K^\circ}(P^\circ,C')$ and $L_2^\circ:=o*\mathcal{C}_{K^\circ}(C',A^\circ)$.
Then we have
\begin{equation*}
\begin{aligned}
\overline{\mathcal{C}_{K^\circ}(P^\circ,C')}&=
\frac{|L_1^\circ|}{\sqrt{1+s^2}}
\begin{pmatrix}
-s \\ 1 \\ 0
\end{pmatrix},  &
\overline{\mathcal{C}_{K^\circ}(R_\ell(C'), P^\circ)}
&=\frac{|L_1^\circ|}{\sqrt{1+s^2}}
\begin{pmatrix}
s \cos \xi+\sin \xi \\ s \sin \xi-\cos \xi \\ 0
\end{pmatrix}, \\
\overline{\mathcal{C}_{K^\circ}(C',A^\circ)}&=
\frac{|L_2^\circ|}{\sqrt{1+s^2}}
\begin{pmatrix}
0 \\ 1 \\ -s
\end{pmatrix}, &
\overline{\mathcal{C}_{K^\circ}(B^\circ, R_\ell(C'))}&=
\frac{|L_2^\circ|}{\sqrt{1+s^2}}
\begin{pmatrix}
\sin \xi \\ -\cos \xi \\ s
\end{pmatrix}, \\
\overline{\mathcal{C}_{K^\circ}(A^\circ,B^\circ)}&=
|M^\circ|
\begin{pmatrix}
0 \\ 0 \\ 1
\end{pmatrix}, &
|L^\circ|&= \frac{|L_1^\circ|+|L_2^\circ|}{\sqrt{1+s^2}}.
\end{aligned}
\end{equation*}
By a direct calculation, the signed volume of $\tilde{K}^\circ$ equals
\begin{equation*}
\begin{aligned}
& \frac{\vq^\circ}{3} \cdot \left(
|L^\circ|
\begin{pmatrix}
 \sin \xi \\ 1-\cos \xi \\ 0
\end{pmatrix}
+
|M^\circ|
\begin{pmatrix}
 0 \\ 0 \\ 1
\end{pmatrix}
+
\frac{s |L_1^\circ|}{\sqrt{1+s^2}}
\begin{pmatrix}
-(1-\cos \xi) \\ \sin \xi \\ 0
\end{pmatrix}
\right) \\
&=
\frac{2{\ell}}{9 |K|}\left(
2(1-\cos \xi) |L|\,|L^\circ| +|M|\,|M^\circ|
\right)
=\frac{\ell}{3|K|}(1-\cos \xi)=\frac{|K^\circ|}{2{\ell}}.
\end{aligned}
\end{equation*}
On the other hand, by the $D_\ell$-symmetry, we have $R_\ell^{k} (\tilde{K}^\circ) \subset K^\circ$ $(0 \leq k \leq \ell-1)$
and the $3$-dimensional Lebesgue measure of $R_\ell^{k} (\tilde{K}^\circ) \cap R_\ell^{k'} (\tilde{K}^\circ)$ is zero for $0 \leq k \neq k' \leq \ell-1$.
Combining the fact that the signed volume of $\tilde{K}^\circ$ is $|K^\circ|/2\ell$ as calculated above,
the sum of the signed volume of the subsets
$\tilde{K}^\circ$, $R_\ell (\tilde{K}^\circ), \dots, R_\ell^{\ell-1} (\tilde{K}^\circ)$ of $K^\circ$
is nothing but the volume of $K^\circ$.
Consequently, we obtain
\begin{equation*}
\tilde{K}^\circ \cup R_\ell (\tilde{K}^\circ) \cup \dots \cup R_\ell^{\ell-1} (\tilde{K}^\circ) = K^\circ,
\end{equation*}
which implies that $\vq^\circ \in \partial K^\circ$.
Thus, the segment $Q^\circ A^\circ$ is on the boundary $\partial K^\circ$.
Since the plane $\{ x=1\}$ contains the points $C^\circ$, $A^\circ$, and $D^\circ$,
the first coordinate of $\vq^\circ$ is greater than or equal to $1$.
On the other hand, since $\va \cdot \vq^\circ \leq 1$,
the first coordinate of $\vq^\circ$ equals $1$.
Since $|L|=1/2$, $|M|=(\sin \xi)/2$, and $|K|=(\ell \sin \xi)/3$, we obtain
\begin{equation*}
\vq^\circ = \frac{\ell}{3 |K|}
\begin{pmatrix}
\sin \xi \\ 1-\cos \xi \\ \sin \xi
\end{pmatrix}
= 
\begin{pmatrix}
1 \\ \frac{1-\cos \xi}{\sin \xi} \\ 1
\end{pmatrix}
=\vd^\circ + 
\begin{pmatrix}
 0 \\ 0 \\ 1
\end{pmatrix}. 
\end{equation*}
By the $D_\ell$-symmetry, $K^\circ$ is the $\ell$-regular right prism with vertices $(R_\ell)^k(Q^\circ)$, $(R_\ell)^k VH(Q^\circ)$ $(1 \leq k \leq \ell)$.
\end{proof}

\subsection{The Case $G=C_{\ell h}$}

\begin{proposition}
Assume that $\ell \geq 3$.
Let $K \in \mathcal{K}^3(C_{\ell h})$. If $\mathcal{P}(K)=\mathcal{P}(P_\ell)$,
then $K$ coincides with $P_\ell$ or $P_\ell^\circ$ up to a linear transformation in $\mathcal{G}$.
\end{proposition}

\begin{proof}
By a similar argument in the proof of Proposition \ref{prop:10},
there exists a sequence $\{K_m\}_{m \in \N} \subset \checkK^3(C_{\ell h})$ such that $K_m \rightarrow K$ $(m \rightarrow \infty)$ with $P, A, B \in \partial K_m$.
For each $K_m$, we have $P^\circ:=\Lambda_m(P)=P$.
However, in this case, we note that $A^\circ=A$ may not hold, where $A^\circ:=\lim_{m \to \infty}\Lambda_m(A)$.
Since $H \va=\va$, we have $H \Lambda_m(\va)=\Lambda_m(\va)$. By $\va \cdot \Lambda_m(\va)=1$, we can represent $\Lambda_m(\va)$ as
\begin{equation*}
 \Lambda_m(\va)=
\begin{pmatrix}
1 \\  a_m \\ 0
\end{pmatrix}
\quad (a_m \in \R).
\end{equation*}
Here, let $\mathcal{R}_\theta$ be the rotation through the angle $\theta$ about the $z$-axis.
Then, for each $m \in \N$ there exists $\theta_m \in [0, 2 \pi)$ such that
\begin{equation*}
 \Lambda_{\mathcal{R}_{\theta_m} K_m}(\va)=
\begin{pmatrix}
1 \\  0 \\ 0
\end{pmatrix}.
\end{equation*}
Taking a subsequence if necessary, we can assume 
$\theta_m \rightarrow \theta$ $(m \rightarrow \infty)$ for some $\theta \in \R$.
Using $\mathcal{R}_{\theta_m} K_m$ and $\mathcal{R}_\theta K$ instead of $K_m$ and $K$,
respectively, we may assume that $\Lambda_m(\va)=\va$.
Thus, we can get the conclusion in a similar way to the proof of Proposition \ref{prop:10}.
\end{proof}

\paragraph{Acknowledgment}

{\small
The first author was supported by JSPS KAKENHI Grant Numbers JP16K05120, JP20K03576.
The second author was supported by JSPS KAKENHI Grant Number JP18K03356.}

\end{document}